\documentclass[journal]{ieeetran}

\let\labelindent\relax

\usepackage{amsmath,amssymb,paralist,subfigure,cite,graphicx,color,amsbsy,amsthm}
\usepackage{stmaryrd,mathrsfs,url}
\usepackage[table]{xcolor}
\usepackage{cuted,mathtools,lipsum}
\usepackage[ruled,vlined,linesnumbered]{algorithm2e}

\usepackage{pifont}
\usepackage[utf8]{inputenc} % allow utf-8 input
\usepackage{hyperref}       % hyperlinks
\hypersetup{
    colorlinks=true,
    linkcolor=cyan,
    filecolor=mnodea,
    urlcolor=cyan,
    citecolor=lime,
}
\usepackage{url}            % simple URL typesetting
\usepackage{booktabs}       % professional-quality tables
\usepackage{amsfonts,amsmath,amssymb}       % blackboard math
\usepackage{amsthm}     % new theorems, definitions
\usepackage{nicefrac}       % compact symbols for 1/2, etc.
\usepackage{microtype}      % microtypography
\usepackage{lipsum}
\usepackage{graphicx}
\usepackage{dsfont}
\usepackage{enumitem}

\usepackage{multicol}
\usepackage{breqn}
\usepackage{stfloats}
\IEEEoverridecommandlockouts
\newtheorem{lem}{Lemma}
\newtheorem{ass}{Assumption}
\newtheorem{thm}{Theorem}

\newtheorem{rem}{Remark}

\def\mb{\mathbf}

\def\mc{\mathcal}

\DeclareMathOperator*{\argmin}{argmin}

\begin{document}
\title{ \bf Discretized Distributed Optimization over Dynamic Digraphs}
% Relaxing Weight-Balance, Bi-Stochasticity, and Beyond}
\author{Mohammadreza  Doostmohammadian,  Wei Jiang, Muwahida Liaquat, Alireza Aghasi,  Houman~Zarrabi
\thanks{M. Doostmohammadian is with the Faculty of Mechanical Engineering at Semnan University, Semnan, Iran (doost@semnan.ac.ir).  W. Jiang and M. Liaquat are with the School of Electrical Engineering at Aalto University, Espoo, Finland (firstname.lastname@aalto.fi). M. Liaquat is also with Finnish Geospatial Research Institute (FGI). A. Aghasi is with the Department of Electrical Engineering and Computer Science, Oregon State University, USA (alireza.aghasi@oregonstate.edu). H. Zarrabi is with Iran Telecom Research Centre (ITRC), Tehran, Iran (h.zarrabi@itrc.ac.ir).  }
}
\maketitle
\thispagestyle{empty}
\pagestyle{empty}

\begin{abstract}
We consider a discrete-time model of continuous-time distributed optimization over dynamic directed-graphs (digraphs) with applications to distributed learning. Our optimization algorithm works over general strongly connected dynamic networks under switching topologies, e.g., in mobile multi-agent systems and volatile networks due to link failures. Compared to many existing lines of work, there is no need for bi-stochastic weight designs on the links. The existing literature mostly needs the link weights to be stochastic using specific weight-design algorithms needed both at the initialization and at all times when the topology of the network changes. This paper eliminates the need for such algorithms and paves the way for distributed optimization over time-varying digraphs. We derive the bound on the gradient-tracking step-size and discrete time-step for convergence and prove dynamic stability using arguments from consensus algorithms, matrix perturbation theory, and Lyapunov theory. This work, particularly, is an improvement over existing stochastic-weight undirected networks in case of link removal or packet drops. This is because the existing literature may need to rerun time-consuming and computationally complex algorithms for stochastic design, while the proposed strategy works as long as the underlying network is weight-symmetric and balanced.  The proposed optimization framework finds applications to distributed classification and learning.
\end{abstract}
\def\abstractname{Note to Practitioners}
\begin{abstract}
	Inspired by recent advances in cloud-computing and distributed and parallel processing along with embedded low-cost CPUs and wireless communications, this paper considers distributed algorithms for optimization and machine learning over wireless-connected autonomous multi-agent systems (MASs). In contrast to the classical centralized learning methods, which are prone to single-point-of-failure and centralized processing, in cooperative optimization the learning is distributed among a group of data-processing agents (e.g. robots) with communication units. This article provides an efficient algorithm to enable MASs to collaboratively optimize a cost function, e.g., for binary classification and distributed support vector machine (D-SVM). Sampled data systems related to MASs and robotic networks, due to digital communications and discretized control models, use discrete-time algorithms that need to account for intermittent communications and dynamic networking. This requires discretized algorithms over \textit{dynamic} digraphs, practical in the presence of packet drops (or lost communication channels). Most existing distributed algorithms either are susceptible to change in the communication network or impose computationally inefficient stochastic design. These algorithms need to be rerun, for example,  whenever the mobile robotic network changes due to limited communication range. This makes most existing algorithms infeasible in real-time applications. Our proposed method outperforms similar algorithms over dynamic robotic networks and in the presence of link removal (packet drops). We show this efficiency of our distributed optimization method by simulation.	
\end{abstract}	
\section{Introduction} \label{sec_intro}

Distributed optimization finds application in many emerging areas from estimation \cite{tase} and machine learning \cite{dsvm} to distributed scheduling and resource allocation \cite{rikos2021optimal,Kar6345156,liu2022distributed,kaheni2021distributed,doostmohammadian2023distributed}. In many applications, the network topology might be time-varying \cite{doan2017scl,dsvm,khan_AB} or contain uni-directional transmission links and data-exchange \cite{gharesifard2013distributed,dsvm,khan_AB,add-opt,pushpull_nedic}. This motivates us to explore optimization strategies over general dynamic directed graphs (digraphs).

This work extends the existing works in the sense that the link weights are \textit{not necessarily bi-stochastic} or split into row and column stochastic matrices \cite{khan_AB,add-opt,pushpull_nedic}. This relaxes the need for weight-stochastic design algorithms \cite{6426252} for weight-symmetric networks in case of link failure or mutual data packet drops. Recall that, any link failure or change in the network topology fails the stochastic condition, and thus, compensation algorithms \cite{6426252,cons_drop_siam,xu2018consensusability,li2019robust,gerencser2018push} are needed to update the link weights to comply with stochasticity in the existing literature. Eliminating these weight constraints allows for the proposed solution to be easily applied over time-varying digraph topologies and networks subject to packet drops.
To prove the convergence of our continuous-time dynamics, we first determine the admissible range of gradient-tracking step-size based on the perturbation analysis \cite{bhatia2007perturbation}. Proving that all the eigenvalues have negative real values (except one zero eigenvalue for each state dimension), the stability is proved; however, the explicit bound on the convergence rate of the algorithm is left for future work\footnote{Our distributed optimization algorithm involves multiple steps like local updates, consensus, communication scheduling, auxiliary variable updates, and more. The explicit interactions between these steps can be complicated to analyze using perturbation theory. Therefore, due to the complexity of perturbation-based analysis and due to the dynamic nature of the network topology it is hard to explicitly find the bound on the convergence rate. }. Next, by discretizing the dynamics, we bound the discrete-time step-size for stability. As compared to some other existing literature, our algorithm operates in a single time-scale with no need for an extra inner consensus loop (as an additional time-scale) \cite{arxiv_digraph,wei_ecc}, and with no need for stochastic weight design \cite{khan_AB,add-opt,pushpull_nedic}. This follows the fact that our weight-balanced condition is more relaxed than the stochastic condition in general. We give particular examples to show improvement over weight-symmetric undirected unreliable networks subject to link removals. We give illustrative examples and simulations to support this claim.

\textbf{General Notations:} $\partial_a(z) = \frac{dz}{da}$ denotes the derivative with respect to $a$. $\sigma(A)$ represents the eigen-spectrum of the matrix $A$.
$I_n$ and $\mb{0}_{n \times n}$ denote the identity matrix and all-zero matrix of size $n$. $\mb{0}_{n}$ and $\mb{1}_{n}$ denote the column vector of all ones or zeros. $\nabla F$ and  $\nabla^2 F$ denote the gradient and second gradient (Hessian) of function $F$. ${\text{blockdiag}[A_i]}$ implies a block-diagonal matrix with each $A_i$ as its $i$th diagonal block. $\otimes$ denotes the Kronecker product.
LHP and RHP are abbreviations of left-half-plane and right-half-plane. The operator norm of a matrix is defined as \cite{bhatia2007perturbation}: $\lVert A\rVert= \sup_{\lVert \mb{x}\rVert=1} \lVert A \mb{x} \rVert $.

\section{Problem Statement} \label{sec_prob}
\subsection{Preliminaries}
We represent the information exchange between nodes by a graph $\mathcal{G}_q=\{\mathcal{V}, \mathcal{E}_q\}$ of order/size $n$, with (time-invariant) set of nodes $\mathcal{V}$ and $\mathcal{E}_q \subseteq \mathcal{V} \times \mathcal{V}$ as the (time-varying) set of links ($q$ as the switching signal).
A directed graph (digraph) is \emph{strongly connected (SC)} if there exists a directed path between every two nodes.
The diameter $d_g$ of $\mathcal{G}$ is the longest shortest path between any two nodes.
%All nodes that can directly transmit information to node $j$  are its in-neighbors $\mathcal{N}^{-}_j$ and $j \in \mathcal{N}^{+}_i$ as their out-neighbor.
Matrices ${W=\{w_{ij}\}}$, ${A=\{a_{ij}\}}$ represent two (generally different) adjacency matrices of~$ \mc{G}$. The Laplacian matrix~$\overline{W}=\{\overline{w}_{ij}^q\}$,  is defined as $\overline{w}_{ij}^q=w_{ij}$ for $i\neq j$ and $\overline{w}_{ii}^q=-\sum_{j=1}^n w_{ij}$ for $i=j$ (similar definition holds for $\overline{A}$). This implies that $\overline{W} \mb{1}_n = 0$ and $\overline{A} \mb{1}_n = 0$.

The work \cite{jadbabaie2003coordination} equivalently defines the Laplacian as $I_n - \overline{D}^{-1}W$ (known as the multi-rate integrator) with the entries of the diagonal (modified) degree matrix $\overline{D}$ as    $\overline{D}_{i,i} = \sum_{j=1}^n w_{ij}+c$ and $c$ as an additive fixed constant to avoid the singularity.
A network is weight-balanced if $\overline{W} \mb{1}_n = \overline{W}^\top \mb{1}_n$ (and similarly $\overline{A} \mb{1}_n = \overline{A}^\top \mb{1}_n$). It is column (resp. row) stochastic if $\overline{W}^\top \mb{1}_n = \mb{1}_n$ (resp. $\overline{W} \mb{1}_n = \mb{1}_n$) and bi-stochastic if both row and column stochastic.
An undirected network with symmetric $W$ (and $A$) is weight-symmetric and weight-balanced. A weight-symmetric and row (or column) stochastic network is also bi-stochastic.

\subsection{Formulation}
The distributed consensus-constraint optimization problem is as follows:
\begin{align}\nonumber
\text{minimize}_{\mb{x} \in \mathbb{R}^{mn}}
F(\mb{x}),\qquad F(\mb x) = \sum_{i=1}^{n} f_i(\mb{x}_i)\\\label{eq_prob} \text{subject to} ~ \mb{x}_1 = \mb{x}_2 = \dots = \mb{x}_n.
\end{align}
Problem \eqref{eq_prob} can be reformulated as
the \textit{Laplacian-constraint} formulation \cite{gharesifard2013distributed},
\begin{align} \label{eq_prob_laplac}
\text{minimize}_{\mb{x} \in \mathbb{R}^{nm}}
 F(\mb x) = \sum_{i=1}^{n} f_i(\mb{x}_i)~~
 \text{subject to}~~\mc{L}\mb{x} = \mb{0}_{nm}.
\end{align}
where $\mc{L} := L \otimes I_m$ and $L$ is the Laplacian of the weight-balanced digraph $\mc{G}$.
Recall from consensus literature \cite{SensNets:Olfati04} that state dynamics in the form $\dot{\mb{x}}=\mc{L}\mb{x}$ over strongly-connected networks results in state consensus $\mb{x}_1=\mb{x}_2=\dots=\mb{x}_n$. This is because, for strongly-connected network, the equilibrium $\mc{L}\mb{x}=\mb{0}_{nm}$ holds if and only if $\mb{x}_1=\mb{x}_2=\dots=\mb{x}_n$ (i.e., consensus is achieved). This gives the intuition why the two formulations are the same. Therefore, for the transition from problem \eqref{eq_prob} to problem \eqref{eq_prob_laplac}, the role or condition of the network and edges is that they must satisfy strong-connectivity. See more details in \cite[Section~III]{gharesifard2013distributed}
Then, under certain conditions (see \cite{gharesifard2013distributed,wei_ecc} for details), problem \eqref{eq_prob} is shown to be equivalent with general \textit{unconstrained} formulation\footnote{Let the optimal state of formlation~\eqref{eq_prob_general} be $\mb{x}^*$, then the optimal state of formulations~\eqref{eq_prob} and \eqref{eq_prob_laplac} is in the form $\mb{1}_n \otimes \mb{x}^*$. } as,
\begin{align}\label{eq_prob_general}
	\text{minimize}_{\mb{x} \in \mathbb{R}^{m}} &
	F(\mb x) = \sum_{i=1}^{n} f_i(\mb{x})
\end{align}
In this work we solve problem~\eqref{eq_prob} and the solution can be easily extended to problems~\eqref{eq_prob_laplac} and~\eqref{eq_prob_general} (similar to \cite{gharesifard2013distributed,wei_ecc}).
\begin{ass} \label{ass_cost}
The local cost functions $f_i: \mathbb{R}^m \mapsto \mathbb{R}$ are smooth, and strictly convex with locally $\gamma$-Lipschitz gradient.
\end{ass}
This assumption implies that each local cost~$f_i: \mathbb{R}^m \mapsto \mathbb{R}$ is (locally) twice differentiable and strictly convex.
Then, the optimal state of \eqref{eq_prob_general} is in the form $\mb{x}^* := \mb{1}_n \otimes \overline{\mb{x}}^*$ with $\overline{\mb{x}}^* \in \mathbb{R}^m$ and ${(\mathbf 1_n^\top \otimes I_m) \boldsymbol{\nabla} F(\mb{x}^*) = \mb{0}_m}$ ~\cite{gharesifard2013distributed}.
For application in D-SVM \cite{dsvm}, we assume that $f_i(\mb{x}_i) = \sum_{j=1}^m f_{i,j}(x_{i,j})$ with $x_{i,j}$ denoting the $j$th element of $\mb{x}_i$.
%Clearly, any solution~$\mb x_i^*,{i=1,\ldots,n}$, of~\eqref{eq_prob} must satisfy ${\sum_{i=1}^{n} \boldsymbol{ \nabla} f_i(\mb{x}^*_i) = \mb{0}_m}$, such that~${\mb{x}^*_1=\ldots=\mb{x}^*_n=\overline{\mb{x}}^*}$, for some~${\overline{\mb{x}}^*\in\mathbb{R}^m}$. In other words, the optimality condition~${\boldsymbol{ \nabla} F(\mb{x}^*) = \mb{0}_{mn}}$ must hold for some~${\mathbf x^*\in\mathbb{R}^{mn}}$ such that~${\mb{x}^*=\mb{1}_n \otimes \overline{ \mb{x}}^*}$, where ${\boldsymbol{\nabla} F:\mathbb R^{mn}\rightarrow\mathbb R^{mn}}$ is the gradient of~${F:\mbb R^{mn}\rightarrow\mbb R}$.

\section{The Proposed Continuous-Time Model}\label{sec_dyn}

\subsection{The Networked Dynamics}
To solve problem~\eqref{eq_prob}, the following continuous-time linear dynamics is proposed:
% for all~$\mb{x}_i(t) \in \mathbb{R}^{m},i\in\{1,\ldots,n\}$,
\begin{align} \label{eq_xdot}	
	\dot{\mb{x}}_i &= -\sum_{j=1}^{n} w^q_{ij}(\mb{x}_i-\mb{x}_j)-\alpha \mb{y}_i, \\ \label{eq_ydot}	
	\dot{\mb{y}}_i &= -\sum_{j=1}^{n} a^q_{ij}(\mb{y}_i-\mb{y}_j) + \partial_t \boldsymbol{\nabla} f_i(\mb{x}_i),
\end{align}
with $n$ as the number of nodes/agents over the network, vector ${\mb{x}_i \in \mathbb R^m}$ as the node $i$'s state, vector~${\mb{y}_i \in \mathbb R^m}$ as an auxiliary variable for gradient-tracking (GT), ${\dot{\mb{x}}_i=\partial_t \mb{x}}_i$, and~${\alpha \in \mathbb{R}_{>0}}$ as the GT step-size. Note that
the general time index is denoted by $t$ and the network switching signal is denoted by index $q$. This switching signal $q:t \mapsto Q$  is a function of time, taking integer values in the set $Q$; in fact, $q(t)$ tells us what \textit{configuration} the graph topology and link weights have at time $t$. $q(t)$ can be periodic or the links can randomly disappear or reappear with a fixed probability. For example, consider $10$ strongly-connected graph topologies $\{\mc{G}_0,\mc{G}_2,\dots,\mc{G}_{9}\}$; the switching signal could be $q=\lfloor \frac{t}{10} \rfloor$ (where $\lfloor x \rfloor$ denotes greatest integer less than or equal to $x$); this switching signal gives the index of the associated graph in the set $Q=\{0,1,\dots,9\}$ at time $t$. The only constraint on these network topologies is that they must be strongly-connected at all time-instants.
%We note that instead of the standard descend direction~$\nabla f_i(\mathbf{x}_i)$, the~$\mathbf{x}_i$-update descends towards an auxiliary variable~${\mb y_i(t)\in\mathbb{R}^m}$, which tracks the sum of local gradients, asymptotically. where~${\dot{\mb{y}}_i=\frac{d{\mb{y}}_i}{dt}}$ and~${A=\{a_{ij}\}}$ is the weighted adjacency matrix with the same zero/non-zero structure as the matrix~$W$. Let~${\mb{y}=[\mb{y}_1;\mb{y}_2;\dots;\mb{y}_n]\in\mbb{R}^{mn}}$ and note that $\frac{d}{d t} \boldsymbol{ \nabla} f_i(\mb{x}_i)=  \boldsymbol{ \nabla}^2 f_i(\mb{x}_i)  \dot{\mb{x}}_i$.
Define the Hessian matrix~${H\coloneqq\text{blockdiag}[\boldsymbol{ \nabla}^2 f_i(\mb{x}_i)]}$, then, Eq. \eqref{eq_xdot}-\eqref{eq_ydot} is written in a compact form
\begin{align} \label{eq_xydot}
	&\left(\begin{array}{c} \dot{\mb{x}} \\ \dot{\mb{y}} \end{array} \right) = M(t,\alpha ) \left(\begin{array}{c} {\mb{x}} \\ {\mb{y}} \end{array} \right),
\\ \label{eq_M}
M(t,\alpha ) = &\left(\begin{array}{cc} \overline{W} \otimes I_m & -\alpha I_{mn} \\ H(\overline{W}\otimes I_m) & \overline{A} \otimes I_m - \alpha H
\end{array} \right).
\end{align}
%with the Laplacian matrix~$\overline{W}=\{\overline{w}_{ij}\}$,  defined as $\overline{w}_{ij}=w_{ij}$ for $i\neq j$ and $\overline{w}_{ij}=-\sum_{i=1}^n w_{ij}$ for $i=j$. The other Laplacian matrix $\overline{A}=\{\overline{a}_{ij}\}$ is defined as $\overline{a}_{ij}=a_{ij}$ for $i\neq j$ and $\overline{a}_{ij}=-\sum_{j=1}^n a_{ij}$ for $i=j$. This implies that $\overline{W}\mb{1}_n = 0$ and $\mb{1}_n^\top \overline{A} = 0$.

We make the following assumption on $ \mc{G}$, $W$ and $A$.
\begin{ass} \label{ass_Wg}
        The graph $\mc{G}$ is directed, and strongly connected at every time $t$. The link weights are positive and $\sum_{j=1}^n w^q_{ij} <1$, $\sum_{j=1}^n a^q_{ij} <1$.
        %, i.e.,  ${w_{ij},a_{ij}\geq0}$, such that ~$\sum_{j=1}^n w_{ij}<1$ and ~${\sum_{j=1}^n a_{ij}<1}$.
		The matrices~${W}$ and~${A}$ are \textit{weight-balanced}.
\end{ass}
	The assumption that $\sum_{j=1}^n w^q_{ij} <1$, $\sum_{j=1}^n a^q_{ij} <1$ is for the sake of proof analysis on bounding $\alpha$ (see Sections~\ref{sec_conv} and \ref{sec_bound}). In case of having $\max\{\sum_{j=1}^n w^q_{ij} ,\sum_{j=1}^n a^q_{ij}\}=c>1$ the $\alpha$ parameter needs to be reduced by a factor of $c$ accordingly.
	Also, the matrices~${W}$ and~${A}$ can be equal (as a special case). The definition of matrices~${W}$ and~${A}$ in the Assumption~\ref{ass_Wg} are relaxed as compared to stochastic features in many existing literature.  This relaxes the stochastic weight design in case of link failure or switching topologies.
	
	The difference between bi-stochastic weight design and balanced weight design is explained here. Recall that bi-stochastic link weights imply that $\sum_{j=1}^n w^q_{ij} = \sum_{i=1}^n w^q_{ij} =1 $, $\sum_{j=1}^n a^q_{ij} = \sum_{i=1}^n a^q_{ij} = 1$. On the other hand, the weight-balanced design is more relaxed as $\sum_{j=1}^n w^q_{ij} = \sum_{i=1}^n w^q_{ij}  $, $\sum_{j=1}^n a^q_{ij} = \sum_{i=1}^n a^q_{ij} $ (there is no need the sum to equal to $1$). This relaxation significantly reduces the complexity of the algorithm, particularly for time-varying network topologies. For example consider the case that a bidirectional symmetric link between nodes $a$, $b$ is removed. Thus, the link weights $w_{ab}=w_{ba}$ are removed while the sum of the weights still remains balanced, i.e., $\sum_{j=1}^n w^q_{ij} - w_{ab} = \sum_{i=1}^n w^q_{ij} - w_{ba}$. However, since the weight-stochasticity does not hold anymore $\sum_{j=1}^n w^q_{ij} - w_{ab} = \sum_{i=1}^n w^q_{ij} - w_{ba} \neq 1$, we need to redesign the rest of the link weights to satisfy bi-stochasticity.
		\begin{rem}
			The existing optimization literature mostly requires bi-stochastic weigh design \cite{khan_AB,add-opt,pushpull_nedic}. One drawback of such a requirement is in the case of link removal when the network loses the bi-stochastic condition. This implies that these existing algorithms do not converge in the case of link removal. Therefore, the weight compensation design algorithms, e.g. \cite{6426252,cons_drop_siam}, are needed to redesign the weights to make the network bi-stochastic again for convergence. This adds more complexity to the solution whenever the network topology changes. Other than the extra complexity, it is likely that the stochastic weight design cannot be even satisfied. Note that not all network topologies have the necessary condition for bi-stochastic weight design. More discussions and illustrations on this are given in Section~\ref{sec_discrete}.
		\end{rem}

% \begin{lem} \label{lem_laplacian}  \cite{SensNets:Olfati04}
%     For SC graph $  \mc{G}$, the eigenvalues of ${\overline{W}}$ and~${\overline{A}}$ are \textit{non-positive}, with algebraic multiplicity of zero eigenvalue equal to $1$.
% For its irreducible Laplacian matrices ${\overline{W}}$ and~${\overline{A}}$, the unit left and right eigenvectors associated with eigenvalue $0$ are respectively $\frac{1}{n}\mb{1}_n^\top$, $\mb{v}_2$,  and $\mb{u}_1^\top$, $\frac{1}{n} \mb{1}_n$ where $\mb{u}_1^\top$ and $\mb{v}_2$ have non-negative elements.
% \end{lem}

\begin{lem} \label{lem_laplacian}  \cite[Theorem~3]{SensNets:Olfati04}
Given SC digraph $  \mc{G}$ of size $n$, the real part of eigenvalues of ${\overline{W}}$ is \textit{non-positive}, with the algebraic multiplicity of zero eigenvalue equal to $1$.
For its irreducible $0$-$1$ structured Laplacian matrix ${\overline{W}}$, consider the right and left eigenvectors $\mb{v}_1$ and $\mb{u}_1^\top$ (including the ones associated with the $0$ eigenvalue) satisfying $\mb{u}_1^\top \mb{v}_1 = 1$. Then, the matrix $R = \mb{v}_1^\top \mb{u}_1$ is real-valued. Further, for the $0$ eigenvalue  $\mb{v}_1 = \frac{1}{\sqrt{n}}\mb{1}_n$.
\end{lem}

% \begin{lem} \label{lem_laplacian}  \cite{SensNets:Olfati04}
%     For SC graph $  \mc{G}$, the eigenvalues of ${\overline{W}}$ and~${\overline{A}}$ are \textit{non-positive}, with algebraic multiplicity of zero eigenvalue equal to $1$.
% For its irreducible Laplacian matrices ${\overline{W}}$ and~${\overline{A}}$, the unit right and left eigenvectors associated with eigenvalue $0$ are respectively $\frac{1}{\sqrt{n}}\mb{1}_n$ and $\mb{u}_1^\top$, $\mb{u}_2^\top$ with non-negative elements.
% \end{lem}
%RECHECK with $\mb{u}_1^\top \mb{v}_1 = 1$

\begin{lem} \label{lem_Wg2}
        For the directed graph $  \mc{G}$ described in Assumption~\ref{ass_Wg} and its associated Laplacian $\overline{W}$, the left eigenvector $\mb{u}_1^\top$ of the $0$ eigenvalue is non-negative. For such a matrix one can assume the left eigenvector satisfies $\sum_{i=1}^n u_{1,i} >0$.
\end{lem}
The above is common in weighted-consensus literature to reach a weighted average of the states, see \cite[Corollary~2]{SensNets:Olfati04}. Recalling  arguments from consensus literature, from Assumption~\ref{ass_Wg} and the disagreement-decay structure of \eqref{eq_xdot}-\eqref{eq_ydot}, we obtain (asymptotically):
\begin{align} \label{eq_sumxdot}
\sum_{i=1}^n \dot{\mb{y}}_i  &= \sum_{i=1}^n \partial_t \boldsymbol{ \nabla} f_i(\mb{x}_i), \\ \label{eq_sumydot}
	\sum_{i=1}^n \dot{\mb{x}}_i
&= -\alpha \sum_{i=1}^n\mb{y}_i.
\end{align}
The above directly follows from the consensus-type structure of the first terms in Eq.~\eqref{eq_xdot}-\eqref{eq_ydot}.
By setting initial values of~$\mb{y}(0)=\mb{0}_{nm}$ and simple integration, we get that $\sum_{i=1}^n \dot{\mb{x}}_i $ tracks $ - \sum_{i=1}^n \boldsymbol{ \nabla} f_i(\mb{x}_i)$ as the gradient sum and eventually reaches zero as discussed next. By setting  ${\dot {\mathbf x}_i =0_m}$, the equilibrium $\underline{\mb{x}}^*$ of the dynamics~\eqref{eq_xdot}-\eqref{eq_ydot}  satisfies~${(\mathbf 1_n^\top \otimes I_m) \boldsymbol{ \nabla} F(\mb{x}^*) = \mb{0}_m}$. This holds for the optimal state of~\eqref{eq_prob} as well~\cite{gharesifard2013distributed}.

\begin{lem} \label{lem_invariant}
	Let~$\mb{y}(0)=\mb{0}_{nm}$ and $\mb{x}(0) \neq \mb{1}_n \otimes  {\mb{x}}_0$, % \notin span{\mb{1}_n}
	for some non-zero~$\mathbf{x}_0\in\mathbb{R}^m$. The global equilibrium of~\eqref{eq_xdot}-\eqref{eq_ydot} includes the states in the form~$[\mb{1}_n \otimes \overline{ \mb{x}}^*;\mb{0}_{nm}]$ (as the invariant set) with $\overline{ \mb{x}}^* \in\mathbb{R}^m$ and includes the optimizer $\mb{x}^*$ of problem \eqref{eq_prob} which satisfies $(\mathbf 1_n^\top \otimes I_m) \boldsymbol{\nabla} F(\mb{x}^*) = \mb{0}_m$.
\end{lem}
\begin{proof}
	From the consensus-type structure of \eqref{eq_xdot}-\eqref{eq_ydot}, the following holds for any~${\mb{x}=\mb{1}_n \otimes \overline{ \mb{x}}^*}, \overline{ \mb{x}}^* \in \mathbb{R}^m$; $\dot{\mb{x}}_i = \mb{0}_m$ from~\eqref{eq_xdot}, and $\dot{\mb{y}}_i = \partial_t \boldsymbol{ \nabla} f_i(\overline{ \mb{x}}^*)=  \boldsymbol{ \nabla}^2 f_i(\overline{ \mb{x}}^*) \dot{\mb{x}}_i = \mb{0}_m$ from~\eqref{eq_ydot}. From \eqref{eq_sumxdot}-\eqref{eq_sumydot}, the following uniquely holds for any~$\mb{x}=\mb{x}^*$ (which is also in the form $\mb{1}_n \otimes \overline{ \mb{x}}^*$),
	\[\sum_{i=1}^n \dot{\mb{x}}_i = -\alpha (\mathbf 1_n^\top \otimes I_m) \boldsymbol{ \nabla} F(\mb{x}^*) =  \mb{0}_m.
	\]
	This shows that~$[\mb{x}^*;\mb{0}_{nm}]$ is an invariant equilibrium under dynamics~\eqref{eq_xdot}-\eqref{eq_ydot} and proves the lemma.
\end{proof}
Lemma~\ref{lem_invariant} defines the invariant set of the networked dynamics~\eqref{eq_xdot}-\eqref{eq_ydot} in the form  $[\mb{x}^*;\mb{0}_{nm}]$. Next, we need to show that the dynamics uniquely converges to this~$\mb{x}^*$ as the optimal point of~\eqref{eq_prob}.

\subsection{Proof of Convergence} \label{sec_conv}
\begin{lem} \label{lem_dM} ~\cite{seyranian2003multiparameter,cai2012average} Consider the $n$-by-$n$ matrix~$P(\alpha)$ as a smooth function of~${\alpha \in \mathbb{R}_{\geq0} }$. Let~$P(0)$ has~${l<n}$ equal eigenvalues~$\lambda_1=\ldots=\lambda_l$, associated with linearly independent  right and left unit eigenvectors~$\mb{v}_1,\ldots,\mb{v}_l$ and~$\mb{u}_1,\ldots,\mb{u}_l$.
Let~${P' = \partial_{\alpha} P(\alpha)|_{\alpha=0}}$ and $\lambda_i(\alpha)$ denote the $i$-th eigenvalue as a function of~$\alpha$, which corresponds to~${\lambda_i, i \in \{1,\ldots,l\}}$. Then,~$ \partial_{\alpha}\lambda_{i}|_{\alpha=0}$ is the $i$-th eigenvalue of, % the following~$l$-by-$l$ matrix,
\[\left(\begin{array}{ccc}
\mb{u}_1^\top P' \mb{v}_1 & \ldots & \mb{u}_1^\top P' \mb{v}_l \\
 & \ddots & \\
\mb{u}_l^\top P' \mb{v}_1 & \ldots & \mb{u}_l^\top P' \mb{v}_l
\end{array} \right).
\]
\end{lem}

\begin{thm} \label{thm_zeroeig}
Let Assumptions~\ref{ass_Wg} and conditions in Lemma~\ref{lem_Wg2} hold. For sufficiently small~$\alpha$, all eigenvalues of~$M$ have non-positive real-parts $\forall t>0 $, and the algebraic multiplicity of the zero eigenvalue is~$m$ (as the individual state dimension).
\end{thm}
\begin{proof}
From~\eqref{eq_M}, let~$M(\alpha)=M_0+\alpha M_1$ with
\begin{eqnarray}\nonumber
	M_0 &=&  \left(\begin{array}{cc} \overline{W} \otimes I_m & \mb{0}_{mn\times mn} \\ H(\overline{W} \otimes I_m) & \overline{A}\otimes I_m \end{array} \right),\\\nonumber
	M_1 &=& \left(\begin{array}{cc} \mb{0}_{mn\times mn} & - {I_{mn}} \\ {\mb{0}_{mn\times mn}} & - H \end{array} \right),
\end{eqnarray}
Since~$M_0$ is block triangular, its eigen-spectrum~$\sigma(\cdot)$  is
\begin{align} \label{eq_sigma}
    \sigma(M_0) = \sigma(\overline{W} \otimes I_m) \cup \sigma(\overline{A} \otimes I_m)
\end{align}
From Lemma~\ref{lem_laplacian}, both matrices ~$\overline{W}$ and~$\overline{A}$ have one isolated zero eigenvalue and the rest in the LHP. The~$m$ sets of eigenvalues of the matrix~$M_0$ (of size $2mn$), associated with dimensions~$j=\{1,\ldots,m\}$ of states~$\mb{x}_i$, then follow as,
$$\operatorname{Re}\{\lambda_{2n,j}\} \leq \ldots \leq \operatorname{Re}\{\lambda_{3,j}\} < \lambda_{2,j} = \lambda_{1,j} = 0,$$
Next, for spectrum analysis, we consider the term~$\alpha M_1$ as a perturbation to $M_0$. Using Lemma~\ref{lem_dM},  one can check how the spectrum of~$M$ changes, as~$M_0$ is perturbed by the values of~$\alpha$ and $ M_1$. In particular, we check the variation of the zero eigenvalues~$\lambda_{1,j}$ and~$\lambda_{2,j}$ of $M_0$  by the perturbation term~$\alpha M_1$. Let~$\lambda_{1,j}(\alpha)$ and~$\lambda_{2,j}(\alpha)$ denote the perturbed eigenvalues associated to $M$. The right and left unit eigenvectors of~$\lambda_{1,j}$,~$\lambda_{2,j}$ follow from Lemma~\ref{lem_laplacian} as\footnote{The normalizing factors of the unit vectors might be ignored as in the followings we only care about the sign of the terms in $U^\top M_1 V$, not the exact values. },
\begin{align} \label{eq_V}
% V &= [V_1~ V_2] =\left(\begin{array}{cc}
% 	\frac{1}{n} \mb{1}_n& \mb{0}_n \\
% 	\mb{0}_n & \mb{v}_2
% \end{array} \right)\otimes I_m
V &= [V_1~ V_2] =
	\frac{1}{\sqrt{n}} \left(\begin{array}{cc}
	\mb{1}_n& \mb{0}_n \\
	\mb{0}_n & \mb{1}_n
\end{array} \right) \otimes I_m
\\ \label{eq_U}
U^\top &= [U_1~ U_2]^\top =\left(\begin{array}{cc}
	\mb{u}_1& \mb{0}_n \\
	\mb{0}_n & \mb{u}_2
\end{array} \right)^\top \otimes I_m
\end{align}
%\[ U^\top =\left(\begin{array}{c}
%	\mb{u}_{1,1}^\top\\ \vdots \\ \mb{u}_{1,m}^\top \\ \mb{u}_{2,1}^\top \\ \vdots \\ \mb{u}_{2,m}^\top
%\end{array}
%\right)=\left(\begin{array}{cc}
%	\mb{1}_n^\top	& \mb{0}_n^\top \\
%	\mb{0}_n^\top	&  \mb{1}_n^\top
%\end{array} \right) \otimes I_m
%\]
These unit eigenvectors satisfy~$V^\top V=I_{2mn}$ and $U^\top U=I_{2mn}$. From the definition ~$\partial_{\alpha}{dM(\alpha)}|_{\alpha=0}=M_1$ and Lemma~\ref{lem_dM},
    	\begin{align} \label{eq_dmalpha}
    		U^\top M_1 V= \left(\begin{array}{cc}
    \mb{0}_{m\times m}	& \times  \\
    \mb{0}_{m\times m}	& -(\mb{u}_2 \otimes I_m)^\top H  (\frac{1}{\sqrt{n}}\mb{1}_n \otimes I_m)
    \end{array} \right).
    \end{align}
From Lemma~\ref{lem_laplacian}, let $u_{2,i}$ denote the $i$th element of $\mb{u}_2$.
    \begin{align} \label{eq_sum_df}
    	-(\mb{u}_2 \otimes I_m)^\top H  (\frac{1}{\sqrt{n}}\mb{1}_n \otimes I_m) = -\sum_{i=1}^n  \frac{u_{2,i}}{\sqrt{n}} \boldsymbol{  \nabla}^2 f_i(\mb{x}_i) \prec 0,
    \end{align}
This follows from the definition of~$H$, Lemma~\ref{lem_Wg2}, and the strict convexity of the cost function in Assumption~\ref{ass_cost}.
From Lemma~\ref{lem_dM}, ~$\partial_{\alpha} d\lambda_{1,j}|_{\alpha=0}$ and~$\partial_{\alpha} d\lambda_{2,j}|_{\alpha=0}$ follow the eigenvalues of the triangular matrix in~\eqref{eq_dmalpha}. This matrix has~$m$ zero eigenvalues and, from \eqref{eq_sum_df},~$m$ negative eigenvalues and, thus,~$\partial_{\alpha}\lambda_{1,j}|_{\alpha=0} = 0$ and~$\partial_{\alpha}\lambda_{2,j}|_{\alpha=0} < 0$. This implies that~$\alpha M_1$ as a perturbation, pushes the~$m$ zero eigenvalues~$\lambda_{2,j}(\alpha)$ of~$M$ toward the LHP while the other $m$ zero eigenvalues~$\lambda_{1,j}(\alpha)$ remain unchanged. Thus, one can conclude that for sufficiently small~$\alpha$, the eigen-spectrum obeys

\small \begin{align}
    \begin{aligned}
         \operatorname{Re}\{\lambda_{2n,j}(\alpha)\} \leq \ldots \leq \operatorname{Re}\{\lambda_{3,j}(\alpha)\}
        \leq \lambda_{2,j}(\alpha) < \lambda_{1,j}(\alpha) = 0,
    \end{aligned}
\end{align} \normalsize
which completes the proof.
\end{proof}

\begin{rem}
 Any pair of Laplacian matrices $\overline{A},\overline{W}$ whose left and right unit eigenvectors satisfy Eq.~\eqref{eq_sum_df} implies that the perturbed eigenvalue moves towards LHP. The only constraint is to satisfy the consensus property on the first terms (on the right-hand) of Eq.~\eqref{eq_xdot}-\eqref{eq_ydot}. In addition, for homogeneous quadratic cost models, we have $H = \gamma I_n$ which is time-invariant. Then, Lemma~\ref{lem_Wg2} can be more relaxed as Eq.~\eqref{eq_sum_df} easily holds for any $\mb{u}_1^\top$ satisfying $\sum_{i=1}^n u_{1,i} >0$.
\end{rem}

The above theorem implies that, for an admissible range of $\alpha>0$, the time-varying matrix~$M$ has only one set of~$m$ zero eigenvalues associated with the eigenvectors~$V_1$ in~\eqref{eq_V}, and its \textit{null space} $\mc{N}(M) = \text{span}\{[\mb{1}_n ;\mb{0}_n ]\otimes I_m\}
$ is \textit{time-independent}.
Recall that the above proof uses the fact that the eigenvalues are continuous functions of the matrix entries~\cite{stewart1990matrix}.
%Theorem~\ref{thm_zeroeig}, similar to \cite{gharesifard2013distributed,ning2017distributed,garg2019fixed2,rahili_ren,taes2020finite,li2020time}, only requires  \textit{strict convexity} of the loss function, as compared to strong convexity in \cite{van2019distributed,nedic2017achieving,ling2013decentralized,simonetto2017decentralized,akbari2015distributed,sun2019convergence,koppel2018decentralized}.
Note that such matrix perturbation analysis as in Theorem~\ref{thm_zeroeig} allows checking the eigen-spectrum of hybrid systems as \eqref{eq_xydot}-\eqref{eq_M} with time-varying system matrices and discrete jumps due to dynamic network topologies.

\subsection{Bounding $\alpha$} \label{sec_bound}
Next, we need to determine the upper bound on $\alpha$ such that no other eigenvalues of $M_0$ under the perturbation moves to the RHP. Some analysis is given in \cite[Lemma~7]{dsvm} based on the notion of optimal matching distance~${d(\sigma(M),\sigma(M_0)) = \min_{\pi} \max_{1\leq i\leq 2nm} (\lambda_i - \lambda_{\pi(i)}(\alpha))}$ with $\pi$ denoting all possible permutations over $2nm$ symbols \cite{bhatia2007perturbation}. This parameter $d(\sigma(M),\sigma(M_0))$ gives the smallest-radius circle centering at and including~$\lambda_{1,j},\ldots,\lambda_{2n,j}$, i.e., the farthest distance between the eigenvalues of~$M$ and~$M_0$. The bound on $\alpha$ is defined via the following lemma.

\begin{lem} \label{lem_dbound} \cite[Theorem~39.1]{bhatia2007perturbation}
Given the system matrix~${M(\alpha) = M_0 +\alpha M_1}$, the optimal matching distance satisfies ${d(\sigma(M),\sigma(M_0))\leq 4(\lVert M_0\rVert+\lVert M\rVert)^{1-\frac{1}{nm}} \lVert \alpha M_1\rVert^{\frac{1}{nm}}}$.
\end{lem}
The sketch of the derivation is shown in Fig.~\ref{fig_perturb}.
\begin{figure}
	\centering
	\includegraphics[width=1.8in]{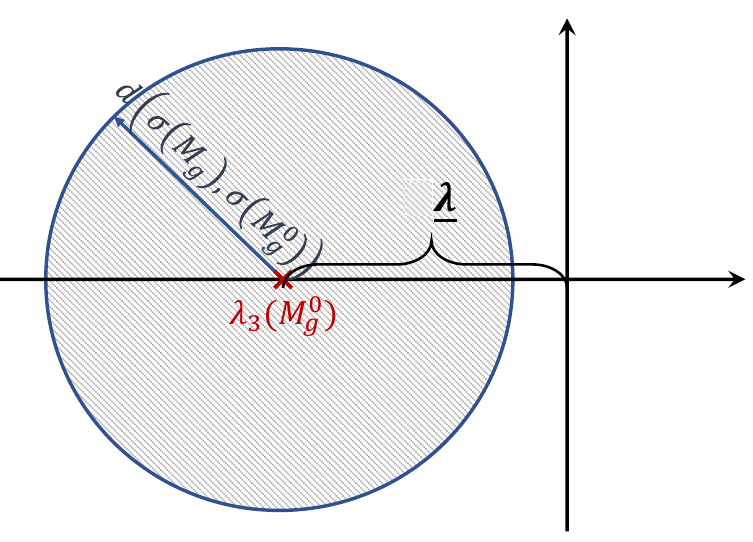}
	\caption{The perturbation analysis to bound $d(\sigma(M),\sigma(M_0))$: the min eigenvalue (in absolute value) of $M_0$ may move towards RHP if $d(\sigma(M),\sigma(M_0))>\underline{\lambda}$.} \label{fig_perturb}
\end{figure}
In order to bound $d(\sigma(M),\sigma(M_0))$ by $\underline{\lambda}$ as the minimum eigenvalue of $M_0$, assuming ${w_{ij},a_{ij}\geq0}$ such that ~$\sum_{j=1}^n w_{ij}<1$ and ~${\sum_{j=1}^n a_{ij}<1}$ (similar to \cite{dsvm}), one can find a bound on $\alpha$ for convergence as follows.

Replacing the spectral norm in Lemma~\ref{lem_dbound} and after some simplification,
\begin{align} \nonumber
 4(|\lambda_{2nm}| + |\lambda_{2nm}|+ \underline{\lambda})^{1-\frac{1}{nm}} \alpha^{\frac{1}{nm}}  \max\{1,\gamma\}^{\frac{1}{nm}} <  \underline{\lambda},
\end{align}
with $\lambda_{2nm}$ denoting the largest (in absolute value) eigenvalue of $M_0$.  This gives
\begin{align} \label{eq_alpha_spec_norm}
0<{\alpha} \leq \frac{ \underline{\lambda}^{nm}}{4^{nm}(2|\lambda_{2nm}| + \underline{\lambda})^{nm-1}   \max\{1,\gamma\}}
\end{align}

The next Lemma borrowed from \cite{SensNets:Olfati04}, is originally known as the \textit{Gresgorin circle Theorem} \cite{hornjohnson} and gives a better perspective on the spectral localization of the eigenvalues of $M_0$,
\begin{lem} \label{lem_gresgorin}\cite[Theorem~2]{SensNets:Olfati04}
Given the Laplacian matrix $\overline{W}$ of a digraph $\mc{G}$, its eigenvalues lie in the following disk (known as the Gresgorin disk) in the complex plane:
\begin{align} \label{eq_gres_disk}
    \mc{D}(\overline{w}_{\max},\overline{w}_{\max}) = \{z \in \mathbb{C}| |z+\overline{w}_{\max}| \leq \overline{w}_{\max} \}.
\end{align}
where $\overline{w}_{\max} = \sum_{j=1,j\neq i}^n |\overline{w}_{ij}|$.
Following the definition of the Laplacian $\overline{w}_{\max} = \max_i \{|\overline{w}_{ii}|\}$.
\end{lem}
Recall that the eigenvalues of $M_0$ coincide with the eigenvalues of $\overline{W}$ and $\overline{A}$, implying that in \eqref{eq_gres_disk}, the radius (and centre) of the Gresgorin disk is $\pm \max_i \{|\overline{w}_{ii}|,|\overline{a}_{ii}|\}$. This is used in the next section to estimate the maximum (absolute) eigenvalue of $M(\alpha)$.
% Now, by adding the $\alpha$-dependent term as $\alpha M_1$, the absolute row sum of $M(\alpha)$ changes by $\max \{\alpha, \gamma (\alpha+2\max_i\{|\overline{w}_{ii}|\}\}$, implying that the eigenvalues are dislocated at most by $2\max \{\alpha, \gamma (\alpha+2\max_i\{|\overline{w}_{ii}|\}\}$. Note that, by adding the $\alpha$-dependent terms, the lower triangular block of $M_0$ also affects the Gresgorin disk. This gives the bound on $\alpha$ as,
% \begin{align} \label{eq_alpha_gres}
% 0<{\alpha} \leq \min \{ \frac{\underline{\lambda}}{2}, \frac{\underline{\lambda} - 2\gamma \max_i\{|\overline{w}_{ii}|\}}{2\gamma}  \}
% \end{align}

Next, we find another admissible $\alpha$ range similar to the analysis in \cite{add-opt}.
First, we recall from \cite[Appendix]{delay_est} to relate the spectrum of $M(t,\alpha )$ in \eqref{eq_M} to $\alpha$. By proper row/column permutations in \cite[Eq.~(18)]{delay_est}, one can find $\sigma(M)$ as, % (RECHECK):

\small \begin{align} \nonumber
\mbox{det}(&\alpha  I_{mn}) \mbox{det}(H(\overline{W}\otimes I_m) +\\ &(\overline{A} \otimes I_m - \alpha H -\lambda I_{mn}) (\frac{1}{\alpha})(\overline{W} \otimes I_m -\lambda I_{mn})) = 0.
\end{align} \normalsize
To make the calculations simpler we let $m=1$ (but it can be extended to any $m>1$),
\begin{align} \label{eq_m=1}
 \mbox{det}(I_{n}) \mbox{det}((\overline{A}  -\lambda I_{n})(\overline{W}  - \lambda I_{n}) +\alpha \lambda H ) = 0
\end{align}
%Similar to the proof of Theorem~\ref{thm_zeroeig}, for stability we need to find the admissible range of  $\alpha$ values for which the eigenvalues $\lambda$ remain in LHP, except one zero eigenvalue. We recall that the eigenvalues are continuous functions of the matrix elements~\cite{stewart_book}.

% Using the above remark and some simplifications, \eqref{eq_m=1}  changes to,
% \begin{align} \label{eq_m=1_sym}
% \mbox{det}((\overline{A}^{sym}_{q}  -\lambda I_{n})(\overline{W}^{sym}_{q}  - \lambda I_{n}) +\alpha \lambda H ) = 0
% \end{align}
% with real eigenvalues $\lambda$.

Note that $\alpha=0$ is a root of \eqref{eq_m=1}, and for small values $\alpha>0$ is admissible for stability. We find the other positive root $\alpha = \overline{\alpha}$ of \eqref{eq_m=1} and, due to continuity of $\sigma(M)$ as a function of $\alpha$ \cite{stewart1990matrix}, we can claim that $\alpha \in (0~\overline{\alpha})$ leads to stability of $M(\alpha)$.
Recall that $\alpha = 0$ gives $\mbox{det}((\overline{W}  - \lambda I_{n})(\overline{A}  - \lambda I_{n})) = 0$ and the eigen-spectrum follows as $\sigma(\overline{A}) \cup \sigma(\overline{W})$ with two zero eigenvalues.

% \tr{To simplify the process of finding $\overline{\alpha}$, first we consider the case $ \overline{W} = \overline{A}^\top$, and we get
% \begin{align} \label{eq_det0}
%  \mbox{det}((\overline{W}  - \lambda I_{n})^\top(\overline{W}  - \lambda I_{n}) + \alpha \lambda H ) = 0
% \end{align}
% Note that the first term (for $\alpha=0$) gives the  set of non-positive eigenvalues $\operatorname{Re}\{\sigma(\overline{W})\}$. This follows from the symmetry of $\overline{W}^\top_{q}\overline{W}$. Let $ \overline{W} = \overline{A}^\top$, then in the last term of \eqref{eq_det_all}$, \overline{W} - \overline{A}^\top$ is a normal and skew-symmetric matrix with zero diagonals and all purely imaginary eigenvalues (as conjugate pairs on imaginary axis).}

% Following Remark~\ref{rem_miror}, we consider the case $ \overline{W} = \overline{A}^\top$, and then Eq.~\eqref{eq_m=1_sym} changes to,
% \begin{align} \label{eq_det0}
%  \mbox{det}((\overline{W}^{sym}_{q}  - \lambda I_{n})^2 + \alpha \lambda H ) = 0
% \end{align}
% Note that the first term (for $\alpha=0$) gives the  set of (non-positive) real eigenvalues $\sigma(\overline{W}^{sym}_{q})$. This follows from the symmetry of $\overline{W}^{sym}_{q}$.

For any $\lambda$ in the LHP, from the diagonal structure of $\alpha H$, one can rewrite \eqref{eq_m=1} (with some abuse of notation) as
\begin{align} \nonumber
 \mbox{det}(&(\overline{A}  -\lambda I_{n} \pm \sqrt{\alpha |\lambda| H})(\overline{W}  - \lambda I_{n} \mp \sqrt{\alpha |\lambda| H} ) \\ &\pm \sqrt{\alpha |\lambda| H}(\overline{W}  -\overline{A})) = 0 \label{eq_det_all}
\end{align}
Letting $\overline{A} = \overline{W}$ to make the last term zero gives,
\begin{align} \nonumber
 \mbox{det}(\overline{W}  &- \lambda I_{n} \mp \sqrt{\alpha |\lambda| H} )= \\ &\mbox{det}(\overline{W}  -\lambda (I_{n}   \mp \sqrt{\frac{\alpha  H}{|\lambda|}} )) = 0
\end{align}
with $\lambda$ as the eigenvalue of $\overline{W}$. The new perturbed eigenvalue (in the above) is $\lambda (1 \pm \sqrt{\frac{\alpha  H}{|\lambda|}})$. One can see that the minimum $\overline{\alpha}$ to make this term zero (towards instability) for a non-zero $\lambda$ is
\begin{align}
    \overline{\alpha} = \argmin_{\alpha} |1 - \sqrt{\frac{\alpha  H}{|\lambda|}}|
%    = \frac{\min \{|\lambda_j|\neq 0\}}{\max \{H_{ii}\}} = \frac{|\lambda_2|}{\gamma}
     \geq \frac{\min \{|\operatorname{Re}\{\lambda_j\}|\neq 0\}}{\max \{H_{ii}\}} = \frac{|\operatorname{Re}\{\lambda_2\}|}{\gamma}
\end{align}
where at the last term we used $H \preceq \gamma I_{n}$.
%This implies that the (real part of) eigenvalues are perturbed towards the RHP by $\sqrt{\alpha |\lambda| H}$ and one can find the other root value $\overline{\alpha}$ from $\lambda I_{n} = \sqrt{\alpha |\lambda| H}$. This gives $\overline{\alpha} \geq \frac{\min \{|\operatorname{Re}\{\lambda_j\}|\neq 0\}}{\max \{H_{ii}\}} = \frac{|\operatorname{Re}\{\lambda_2\}|}{\gamma}$ for $H \preceq \gamma I_{n}$.
Then, from perturbation analysis as in Theorem~\ref{thm_zeroeig}, one can claim the stability range as
%(RECHECK $\frac{|\operatorname{Re}\{\lambda_2\}|}{\gamma}$)
\begin{align} \label{eq_alphabar0}
    0 < \alpha < \overline{\alpha}:= \frac{|\operatorname{Re}\{\lambda_2\}|}{\gamma}
\end{align}
for which $M(\alpha)$ has only one isolated zero eigenvalue and the rest with negative values. Note that having $H = \gamma I_n$, as in homogeneous quadratics models for CPU balancing \cite{rikos2021optimal} or power scheduling \cite{Kar6345156}, the above gives a tighter bound on $\alpha$.
In general, for $\overline{W} \neq \overline{A}$ matrices, the admissible range changes to,
% \begin{align} \label{eq_alphabar}
%     0 < \alpha < \overline{\alpha}:= \frac{\sqrt{| \lambda_2(\overline{W}) \lambda_2(\overline{A}) | }}{\gamma}
% \end{align}
\begin{align} \label{eq_alphabar}
    0 < \alpha <  \frac{\min \{|\operatorname{Re}\{\lambda_2(\overline{W})\}|,| \operatorname{Re}\{\lambda_2(\overline{A})\}|\}}{\gamma} =: \overline{\alpha}
\end{align}
This is because choosing the diagonal blocks of $M_0$ as both equal to either $\overline{A}$ or $\overline{W}$ gives $\overline{\alpha}:= \frac{|\operatorname{Re}\{\lambda_2(\overline{A})\}|}{\gamma}$ or $\overline{\alpha}:= \frac{|\operatorname{Re}\{\lambda_2(\overline{W})\}|}{\gamma}$ in \eqref{eq_alphabar}, respectively.
This upper-bound on $\alpha$ is oblivious to the size of the system and adjacency matrices (only depends on their eigen-spectrum) and holds for any value of $m \geq 1$.

In general, the stability analysis of hybrid systems as \eqref{eq_xydot}-\eqref{eq_M} are challenging (see some relevant discussions in \cite{cortes2008discontinuous} and also,  on  \textit{hybrid consensus} setups, in \cite[Section~IX]{SensNets:Olfati04}). The above perturbation-based analysis is the key component of the convergence proof along with the proper choice of the Lyapunov function as in \cite{dsvm,SensNets:Olfati04}. The stability is discussed next in the main theorem of this paper.

% RECHECK
% \textcolor{red}{
% \begin{rem} \label{rem_courant}
% The equivalent of \textit{Courant-Fischer} Theorem (for \textit{symmetric} matrices) can be stated for general \textit{non-symmetric} matrices by considering the \textit{singular} values instead of eigenvalues \cite{Spielman} and \cite[Corollary~6.2]{bhatia2007perturbation}. For a matrix $M$,
% \begin{align} \label{eq:cour_fischer}
% \min_{\mb{1}_n^\top \mb{x} = 0, \mb{x} \neq \mb{0}_n}
% \frac{ |\mb{x}^\top M^\top M \mb{x}|}{\mb{x}^\top \mb{x}} = \lambda_2 (M^\top M) \geq |\operatorname{Re}\{{\lambda}_{2,j}(\alpha)\}|
% \end{align}
% The last inequality simply follows from
% $$ \lambda_\min (L^\top L) \geq  \lambda_\min (L^\top ) \times  \lambda_\min (L) \geq \operatorname{Re}\{{\lambda}_{2,j}(\alpha)\}^2
% $$
% %min(eig($L^\top L$)) $\geq $ min(eig($L^\top$)) * min(eig($L$)) $\geq $ real($\lambda_2$)^2
% \end{rem}
% }

\begin{thm} \label{thm_lyapunov}
Let Assumptions \eqref{ass_cost}-\eqref{ass_Wg} and conditions in Lemma~\ref{lem_Wg2}, Lemma~\ref{lem_invariant}, and Theorem~\ref{thm_zeroeig} hold. For $\alpha$ satisfying \eqref{eq_alphabar}, the  dynamics~\eqref{eq_xdot}-\eqref{eq_ydot} (and its matrix form \eqref{eq_xydot}-\eqref{eq_M}) converges to the invariant equilibrium~$[\mb{x}^*\otimes \mb{1}_n;\mb{0}_{nm}]$, where~$\mb{x}^*\otimes \mb{1}_n$ denotes the optimizer of the problem~\eqref{eq_prob}.
\end{thm}
\begin{proof}
From Lemma~\ref{lem_invariant},~$[\mb{x}^*\otimes \mb{1}_n;\mb{0}_{nm}]  \in \mathbb{R}^{2mn}$ with $\mb{x}^* \in \mathbb{R}$ is the invariant set under dynamics~\eqref{eq_xydot}-\eqref{eq_M} and belongs to $\mc{N}(M)$. From Theorem~\ref{thm_zeroeig} and for an admissible range of  $\alpha \in (0~\overline{\alpha})$, we showed that the eigenvalues of $M(\alpha,t )$ remain in the LHP except for the zero eigenvalue of multiplicity one. This is \textit{irrespective of the network topology $ \mc{G}$ and the change in $M(\alpha,t )$} which may even allow for possible hybrid stability analysis \cite{goebel2009hybrid}. Following \cite[Corollary~1]{SensNets:Olfati04}, this network of integrators is asymptotically globally stable. On the other hand, its equilibrium is associated with the (eigenspace of) right eigenvectors associated with its zero eigenvalue of multiplicity $m$, which is a set of $m$ decoupled vectors as $\{[\mb{1}_n ;\mb{0}_n ]\otimes I_m\}$. This simply implies that the invariant set $[\mb{x}^*\otimes \mb{1}_n;\mb{0}_{nm}]$ is the equilibrium under dynamics~\eqref{eq_xydot}-\eqref{eq_M}.
\end{proof}
For undirected networks, since the eigenvalues of $M$ are real, one extends the proof to hybrid setups under a switching signal $q$.

% \begin{rem} \label{rem_miror}
% In our previous work \cite{dsvm}, we showed that $\alpha$ values along with other relevant network parameters (e.g., its diameter) determine the convergence rate of the algorithm. Using the notion of mirror operation on consensus digraphs, one can extend this to generally balanced directed networks. This mirror operation implies that for balanced digraphs one can define graphs with \textit{symmetric} adjacency matrices $W^{sym}_q = \frac{W+W^\top}{2}$, $A^{sym}_q = \frac{A+A^\top}{2}$. Then, the convergence (decay) rate of the disagreement function associated with $W$ and $A$ follow the spectral properties of graph Laplacians $\overline{W}^{sym}_{q}$, $\overline{A}^{sym}_{q}$ and similar to the undirected graphs with real eigenvalues. See more details in \cite[Section~V]{SensNets:Olfati04}. For admissible and sufficiently small $\alpha$, the perturbed eigenvalues of $M(\alpha)$ then remain in the range $d(\sigma(M),\sigma(M_0))$ of $\sigma(M_0)$.
% \end{rem}

\section{The Discrete-Time Model: Improvement over the Existing Models} \label{sec_discrete}
\subsection{Discrete-time version}
For agents with discrete-time models, one can follow similar discussions as for discretized consensus dynamics in \cite[Section~IV]{SensNets:Olfati04}. The discretized version of the CT dynamics~\eqref{eq_xydot} takes the following form,
\begin{align} \label{eq_xydot_d}
	\left(\begin{array}{c} \mb{x}(k+1) \\ \mb{y}(k+1) \end{array} \right) = M_d(\eta,\alpha ) \left(\begin{array}{c} {\mb{x}(k)} \\ {\mb{y}(k)} \end{array} \right),
\end{align}
% \begin{align} \label{eq_Md}
% M_d(t,\alpha ) = \left(\begin{array}{cc} \overline{W} \otimes I_m & -\alpha I_{mn} \\ H(\overline{W}\otimes I_m) & \overline{A} \otimes I_m - \alpha H
% \end{array} \right).
% \end{align}
%In this section, we discuss the DT approximation of the CT dynamics~\eqref{eq_xydot}-\eqref{eq_M}
with~$k$ as the discrete time-step. The most common approximation for~$M_d$ is by Euler Forward Method (EFM)
as~$M_d = I+\eta M$
with~$\eta$ as the (discrete) sampling step-size. This implies that the discretized version of the node dynamics \eqref{eq_xdot}-\eqref{eq_ydot} is in the following form:
\begin{align} \label{eq_xdot_d}	
	\mb{x}_i(k+1) = \mb{x}_i(k) - &\eta \sum_{j=1}^{n} w^q_{ij}(\mb{x}_i(k)-\mb{x}_j(k))-\alpha \mb{y}_i(k), \\ \nonumber	
	\mb{y}_i(k+1) = \mb{y}_i(k) -&\eta \sum_{j=1}^{n} a^q_{ij}(\mb{y}_i(k)-\mb{y}_j(k)) \\ \label{eq_ydot_d}
	&+ \boldsymbol{\nabla} f_i(\mb{x}_i(k+1))-\boldsymbol{\nabla} f_i(\mb{x}_i(k)),
\end{align}

As discussed in Section~\ref{sec_conv}, matrix~$M$ has~$m$ zero eigenvalues while all other eigenvalues have negative real-part. This implies that~$M_d$ given by EFM has~$m$ eigenvalues at~$1$ with $m$ independent (decoupled) eigenvectors, while for stability, its other eigenvalues need to satisfy~$|\lambda_i|<1$ which depends on the sampling step~$\eta$.
The explicit upper-bound for~$\eta$ such that stable CT dynamics remains stable after discretization via first-order Euler approximation is given, for example, in~\cite[Table~I]{axelsson2014discrete} as,
\begin{align}\label{eq_Tbound}
	\eta<\min_{2\leq i \leq 2n,1\leq j\leq m} \frac{2 |Re\{\lambda_{i,j}(\alpha)\}|}{|\lambda_{i,j}(\alpha)|^2},
\end{align}
with~$\lambda_{i,j}(\alpha)$ as the eigenvalue of the time-varying matrix~$M$. For real Laplacian matrices associated with linear consensus protocols, e.g., $\overline{W}$
%or having $ \overline{W} = \overline{A}^\top$,
the discretized model is given as $P_w = I_n+\eta \overline{W}$ which satisfies the following Lemma.

\begin{lem} \cite{Olfati_proceedings} \label{lem_olfati}
   Given $P_w = I_n+\eta \overline{W}$ with $\overline{w}_{\max} = \sum_{j=1,j\neq i}^n |\overline{w}_{ij}|$ (from Lemma~\ref{lem_gresgorin}), $\overline{W}$ as the Laplacian matrix satisfying Assumption~\ref{ass_Wg}, and $0< \eta < \frac{1}{\overline{w}_{\max}}$, the followings hold:
   \begin{itemize}
       \item $P_w$ is row-stochastic and non-negative with one isolated eigenvalue at $1$.
       \item The eigenvalues of $P_w$ are inside the unit circle.
       \item $P_w$ is primitive, i.e., it has only one eigenvalue with maximum modulus.
   \end{itemize}
\end{lem}
%The following lemma is then useful for stability analysis.
% \begin{lem} \cite[Lemma~3.7]{ren1431045} \label{lem_ren}
%   Let $P_w$  be stochastic with one isolated eigenvalue at $1$ and all the other eigenvalues inside the unit circle. Then, $P_w$ is SIA, i.e., $\lim_{l \rightarrow \infty} P_w^l = \mb{1}_n \mb{v}_w^\top$ where (the non-negative) $\mb{v}_w$ is the right eigenvector associated to $1$.
% \end{lem}

One can state Lemma~\ref{lem_olfati}
similarly for $\overline{A}$. We use this for the stability analysis of the DT version~\eqref{eq_xydot_d}  via a similar procedure as in Section~\ref{sec_conv} by proving that~$M_d$ has~$m$ eigenvalues at~$1$ with all other eigenvalues within the unit circle. Assuming~${M_d = (I+\eta M_0)+\alpha \eta M_1}$, it is clear that for~$\eta$ satisfying~Lemma~\ref{lem_olfati} all the eigenvalues of~$I+\eta M_0$ remain inside the unit circle except for~$2m$ eigenvalues~$\lambda_{1,j}=\lambda_{2,j} =1$ (from Lemma~\ref{lem_olfati}). In this case, we consider~$\eta \alpha M_1$ as a perturbation to~$I + \eta M_0$ and we follow a similar analysis as in the proof of Theorem~\ref{thm_zeroeig}. Note that the eigenvectors associated with the $1$ eigenvalues are the same as \eqref{eq_V}-\eqref{eq_U}. Then, from Lemma~\ref{lem_dM} and following similar equation to derive Eq.~\eqref{eq_sum_df}, we have
$\frac{d\lambda_{1,j}}{d\alpha}|_{\alpha=0} = 0$ and~$\frac{d\lambda_{2,j}}{d\alpha}|_{\alpha=0}<0$. This implies that~$M_d$ has~$m$ eigenvalue at~$1$ (with the same right eigenvectors $V_1$) and for sufficiently small~$\alpha \eta$ all other eigenvalues remain inside the unit circle. To bound~$\alpha \eta$, we follow a similar procedure as in Section~\ref{sec_bound}.
From Lemma~\ref{lem_dbound}, we need $d(\sigma(M_d),\sigma(I+\eta M_0))<1- \lambda_{\max} $ where~$\lambda_{\max} =\max_{1\leq j\leq m, 2\leq i\leq n }|\lambda_{i,j}|$ as the eigenvalue (with its maximum absolute value less than $1$) of~$I + \eta M_0$, and the bound on~$\eta \alpha$ can be determined.
From Lemma~\ref{lem_gresgorin}, for $0$-$1$ adjacency matrix $W$, $\overline{w}_{\max}$ denotes the max node degree \cite{SensNets:Olfati04}, and the algebraic connectivity of $\overline{W}$ satisfies $\lambda_2(\overline{W}) \geq \frac{1}{nd_g}$ with $d_g$ as the network diameter \cite[p.571]{graph_handbook}. This gives an idea of the eigen-spectrum of $M_0$ and the discretized matrix $I+\eta M_0$. Using perturbation analysis, one can find an estimate on the eigen-spectrum of $M_d(\alpha)$ by considering the perturbation term as $\eta \alpha M_1$. Then, Eq.~\eqref{eq_alpha_spec_norm} changes to the following which gives one admissible range of $\alpha \eta$,
\begin{align} \label{eq_alpha_eta_spec}
0< \alpha \eta \leq \frac{ (1-\lambda_{\max})^{nm}}{4^{nm}(3 -\lambda_{\max})^{nm-1}   \max\{1,\gamma\}}
\end{align}
Similar to the derivation resulting to Eq.~\eqref{eq_alphabar}, one can find the min $\overline{\alpha}$ towards instability for $|\lambda| \neq 1$ as
\begin{align} \nonumber
    \overline{\alpha} &= \argmin_{\alpha} |1 - \sqrt{\frac{\alpha \eta H}{|\lambda - 1|}}| \\
     &\geq \frac{1-\lambda_{\max}}{\eta \max \{H_{ii}\}} = \frac{1-\lambda_{\max}}{\eta \gamma}
\end{align}
and the bound in Eq.~\eqref{eq_alphabar} changes to,
\begin{align} \label{eq_alphabar_eta}
    0 < \alpha \eta < \frac{\min \{1 - \lambda_{\max}(\overline{A}) , 1 - \lambda_{\max}(\overline{W}) \}}{\gamma}
\end{align}
where the bound is on both $\alpha \eta$ as the perturbation parameter.
For undirected networks, i.e., symmetric $\overline{W}$, $\overline{A}$ with real eigenvalues and eigenvectors, for simplicity assume $m=1$ where the results can be easily extended to $m>1$. Three cases may occur for $\lambda_{\max}(\overline{W}), \lambda_{\max}(\overline{A})$ which are largest (in absolute value) eigenvalues inside the unit circle: (i) $|\lambda_{\max}(\overline{W})| > |\lambda_{\max}(\overline{A})|$ with vector $[\mb{v}_w; \mb{0}_n]$ as the right eigenvector of $M_0$ associated to it. (ii) $\lambda_{\max}(\overline{W}) < \lambda_{\max}(\overline{A})$ with  $[\mb{0}_n; \mb{v}_a]$ as its right eigenvector. (iii) $\lambda_{\max}(\overline{W}) = \lambda_{\max}(\overline{A})$ (algebraic multiplicity of two) with $[\mb{v}_w, \mb{0}_n; \mb{0}_n, \mb{v}_a]$ as the right eigenvectors. Then, from Lemma~\ref{lem_dM} and following similar equation to derive Eq.~\eqref{eq_sum_df}, for case (i),
\begin{align} \label{eq_dmalpha1}
    [ \mb{v}_w; \mb{0}_n]^\top M_1 [ \mb{v}_w; \mb{0}_n]= 0.
\end{align}
which implies that, after perturbation, we have $ \partial_{\alpha} \lambda_{\max}|_{\alpha=0} = 0$. For case (ii),
\begin{align} \label{eq_dmalpha2}
    [ \mb{0}_n; \mb{v}_a]^\top M_1 [ \mb{0}_n; \mb{v}_a]= -\sum_{i=1}^n \mb{v}^2_{a,i} \partial_{\mb{x}_i} f_i < 0.
\end{align}
which implies that perturbation leads to $ \partial_{\alpha}\lambda_{\max}|_{\alpha=0} < 0$ and $\lambda_{\max}<1$ moves further inside the unit circle and gets smaller in absolute value. For case (iii) as a combination of the other two cases,
\begin{align} \label{eq_dmalpha3}
    \left(\begin{array}{cc}
    \mb{v}_w & \mb{0}_n  \\
    \mb{0}_n	&  \mb{v}_a
    \end{array} \right)^\top M_1
    \left(\begin{array}{cc}
    \mb{v}_w & \mb{0}_n  \\
    \mb{0}_n	&  \mb{v}_a
    \end{array} \right) = \left(\begin{array}{cc}
    0	& \times  \\
    0	& -\mb{v}_a^\top H \mb{v}_a
    \end{array} \right).
\end{align}
saying that $ \partial_{\alpha}\lambda_{\max}(\overline{W})|_{\alpha=0} = 0$, $ \partial_{\alpha}\lambda_{\max}(\overline{A})|_{\alpha=0} < 0$. This implies that $\lambda_{\max}(\overline{A})$ gets smaller in size, and moves further towards inside the unit circle and $\lambda_{\max}(\overline{W})<1$ remains constant. Since the perturbation parameter is $\alpha \eta$ in Eq.~\eqref{eq_alphabar_eta}, $\alpha$ needs to be decreased for larger values of the sampling period $\eta$. Following Lemma~\ref{lem_dbound}, one can claim that $|\lambda_{\max}(\alpha)|  \geq |\lambda_{\max}| + d(\sigma(M(\alpha)),\sigma(M_0))$ with $|\lambda_{\max}| := \max_{2\leq i \leq 2n,1\leq j\leq m} |\lambda_{i,j}|$ (for $\alpha = 0 $). For a given $0 <\alpha < \overline{\alpha}$, then Eq.~\eqref{eq_Tbound} gives another bound as,

\scriptsize \begin{align} \label{eq_Tbound_d}
	\eta < \frac{1}{  \max\{\overline{w}_{\max}, \overline{a}_{\max}\} +2(\lVert M_0\rVert+\lVert M(\alpha)\rVert)^{1-\frac{1}{2mn}} \lVert \alpha M_1\rVert^{\frac{1}{2mn}} },
\end{align} \normalsize
ensures stable system mapping from CT to DT and convergence under the discretized model \eqref{eq_xydot_d}. We summarized our proposed distributed optimization solution in Algorithm~\ref{alg_1}.

\begin{algorithm} [t] \label{alg_1}
		\textbf{Given:} costs $f_i(\mb{x}_i)$, dynamic graph $\mc{G}_q$, weight matrices $W_q,A_q$, switching signal $q$  \\	
		\textbf{Initialization:} ${\mb{y}}_i(0)=\mb{0}_{m}$, initialize ${\mb{x}}_i(0)$  with random values \;
		Choose $\eta$ satisfying Eq.~\eqref{eq_Tbound_d} and $0<\alpha<\overline{\alpha}$ satisfying Eq.~\eqref{eq_alphabar}
		\\
		\For{$k<T_{termination}$}{
			Node $i$ receives $\mb{x}_j(k)$ and $\mb{y}_j(k)$ data from the neighbors $j$\;
			Node $i$ finds $\mb{x}_i(k+1)$ via discrete dynamics \eqref{eq_xdot_d}\;
			Node $i$ finds $\boldsymbol{ \nabla} f_i(\mb{x}_i(k+1))$\;
			Node $i$ finds  $\mb{y}_i(k+1)$ via discrete dynamics \eqref{eq_ydot_d}\;
			Node $i$  shares $\mb{x}_i(k+1)$ and $\mb{y}_i(k+1)$ over $\mc{G}_q$\;
			Set $k \leftarrow k+1$
			}
		\textbf{Return:} final optimal state $\mb{x}^*$ and final cost $F(\mb{x}^*)$\;	
		\caption{The proposed distributed optimization algorithm. }
\end{algorithm}

\subsection{Performance Subject to Link Removal}
The proposed discrete strategy shows better performance over dynamic networks and switching topologies. This is because the WB condition for matrices $W$ and $A$ (Assumption~\ref{ass_Wg}) is much easier to satisfy as compared to weight-stochastic design algorithms, e.g., in existing discrete-time literature \cite{khan_AB,add-opt,pushpull_nedic}. Note that many existing works do not work under switching network topologies or need weight-redesign algorithms to satisfy stochastic conditions \cite{6426252,cons_drop_siam} in case of link failure. Since in general weigh-balancing algorithms \cite{2014:ISCCSP2} are computationally more efficient than stochastic design algorithms and in some cases, it is not even possible to satisfy the weight stochasticity for some networks (see example in Fig.~\ref{fig_remov}-(Right)), the proposed strategy outperforms the existing literature. This specifies a case where there is no need to redesign the weights for the WB condition (for weight-symmetric undirected networks) but needs the weights to be redesigned (if possible) for the stochastic condition. Some other examples in distributed resource allocation and coupling-constrained distributed optimization subject to packet drops and link removals are given in \cite{icrom}. In \cite{icrom} it is claimed that such a WB condition is more preferred in real-time applications as it requires no need to redesign stochastic algorithms in case the network is unreliable and dynamic, e.g., in packet loss scenarios.

\begin{rem}
	For general \textit{symmetric WB undirected graphs} with unreliable links, the WB conditions hold after link removal/failure, but the bi-stochastic condition does not necessarily hold for the same networks after link removal. For example, consider the illustration in Fig.~\ref{fig_remov}.
\end{rem}

\begin{figure}
	\centering
	\includegraphics[width=1.5in]{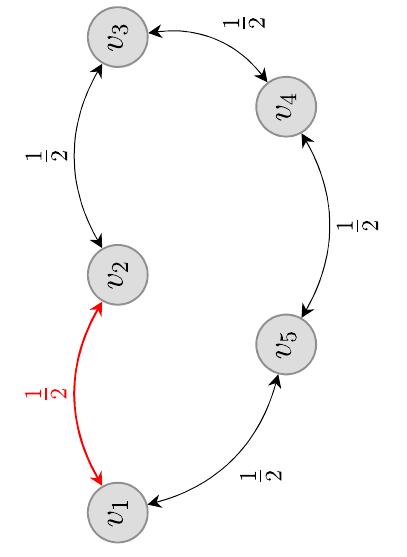}
	\includegraphics[width=1.5in]{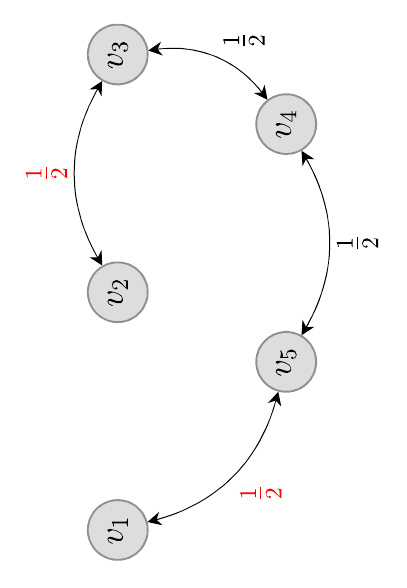}
	\caption{(Left)-An undirected unreliable network of $5$ nodes which is both stochastic and WB. The red link represents an unreliable link that might be subject to failure.  (Right)-The same network after removal/failure of the (bi-directional) unreliable link: the network is still WB, but is not bi-stochastic anymore. To apply the existing weight-stochastic literature, e.g., \cite{khan_AB,add-opt,pushpull_nedic}, one needs to redesign the weights, e.g., by weight compensation algorithms in \cite{6426252,cons_drop_siam}, to redesign the weights if possible and make the new topology weight-stochastic again.  } \label{fig_remov}
\end{figure}

\section{Simulation: Nonlinear SVM}
\begin{figure}
	\centering
	\includegraphics[width=1.6in]{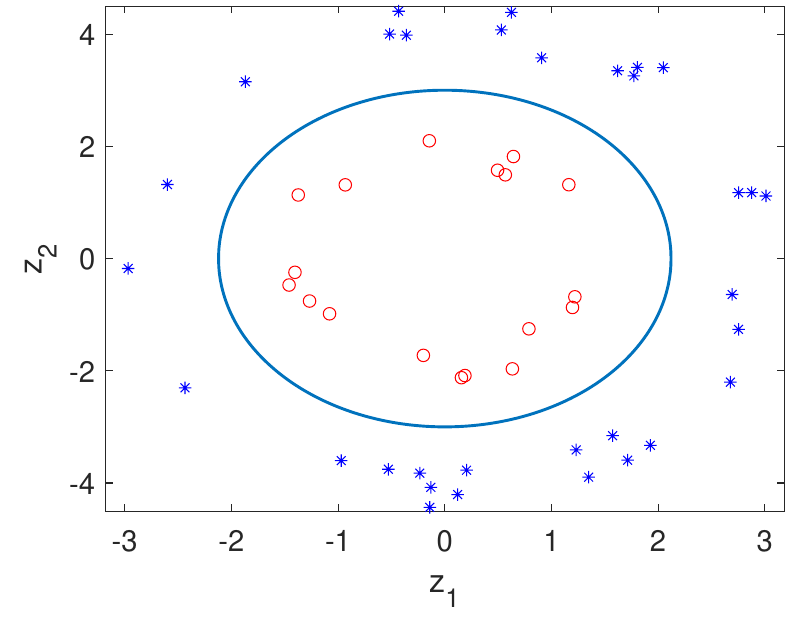}
	\includegraphics[width=1.65in]{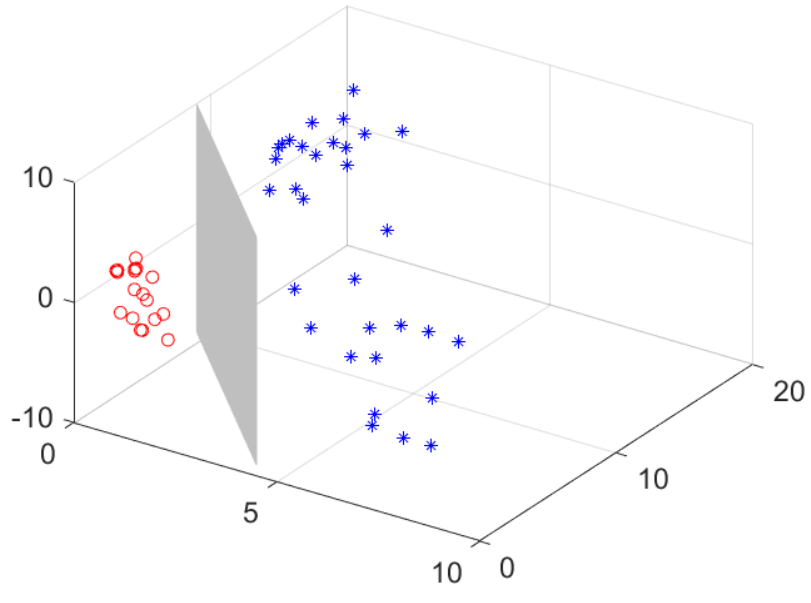}
	\caption{This figure shows the data points scattered in 2D  and classified by the SVM hyperplane in 3D. } \label{fig_graph}
\end{figure}
\begin{figure*} [bpt]
	\centering
	\includegraphics[width=2.2in]{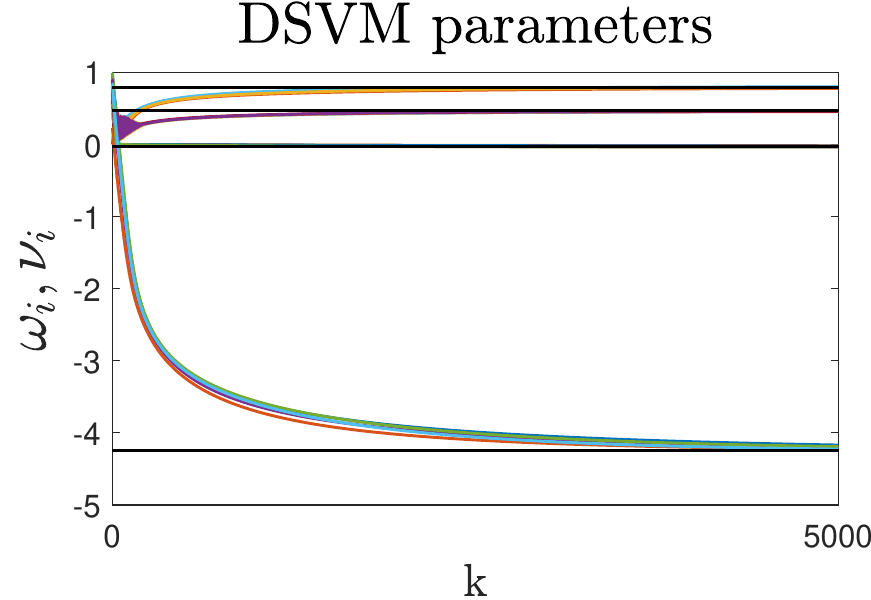}
	\includegraphics[width=2.3in]{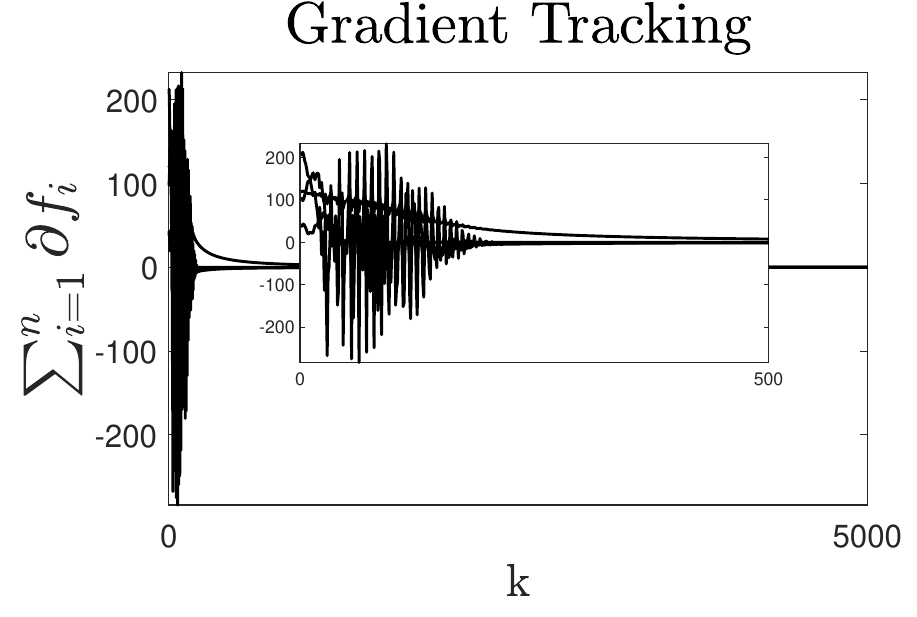}
	\includegraphics[width=2.3in]{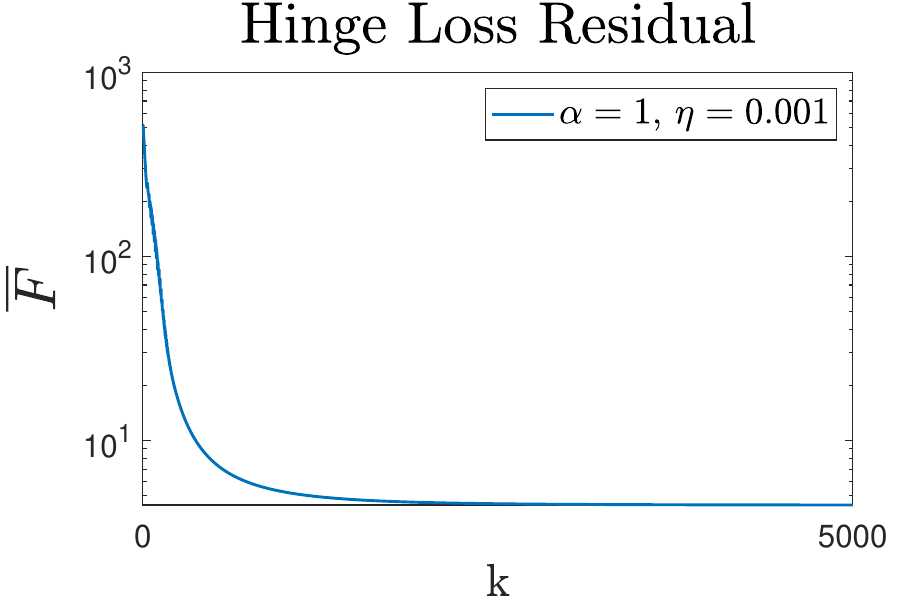}
	\includegraphics[width=2.2in]{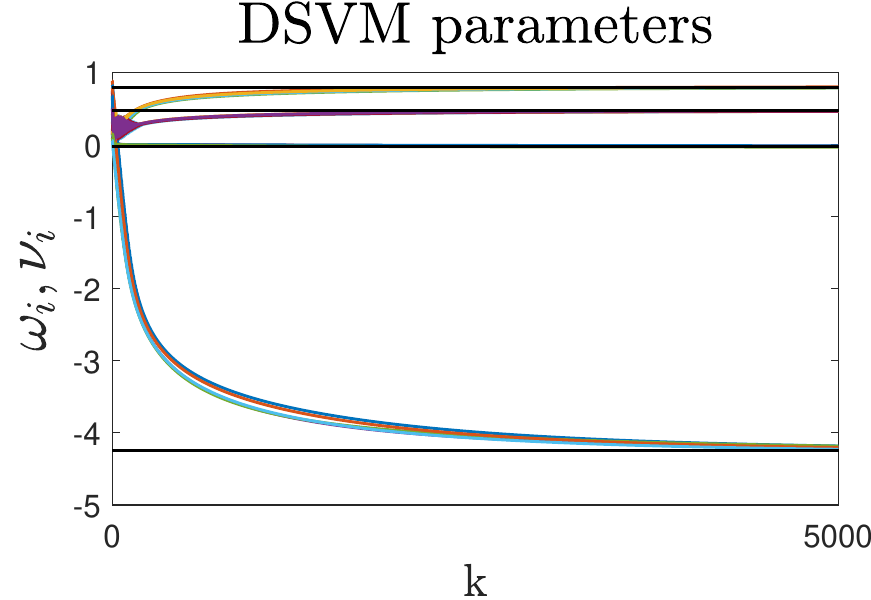}
	\includegraphics[width=2.3in]{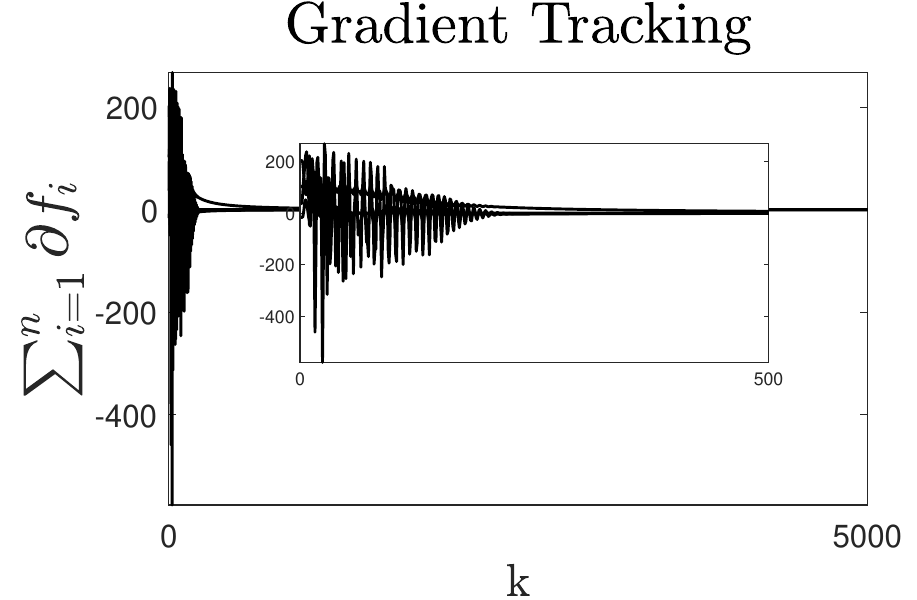}
	\includegraphics[width=2.3in]{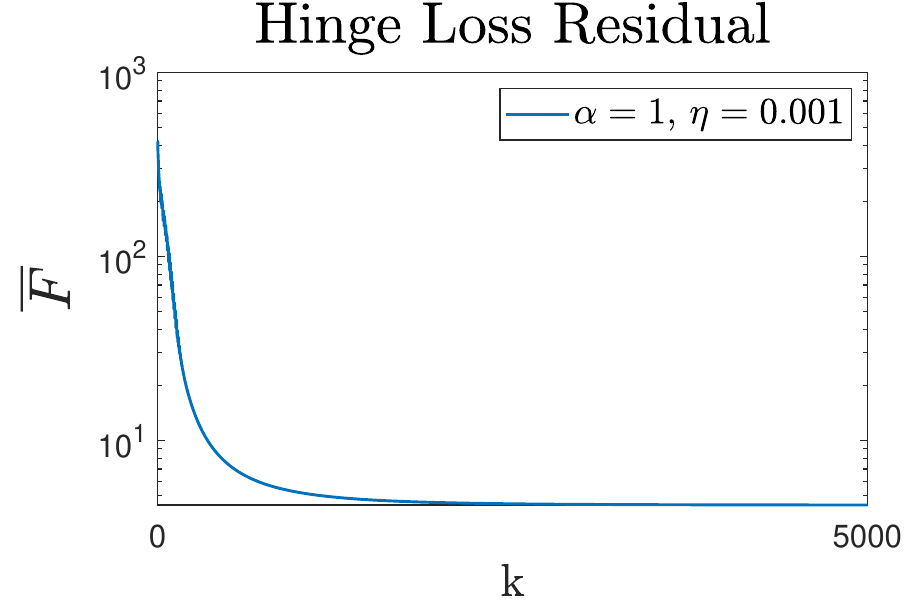}	
	\caption{The SVM parameters~${\omega}_i$ and~$\nu_i$ (at all $n=5$ nodes in Fig.~\ref{fig_remov}) under discrete dynamics~\eqref{eq_xydot}-\eqref{eq_M}. The top row represents the parameters before link removal and the bottom row is for after link removal. The gradient sum~$\sum_{i=1}^{n} \boldsymbol{ \nabla} f_i(\mb{x}_i)$ is also shown to admit the GT property.  The hinge loss residual~$\overline{F}(\mb{x}) = \sum_{i=1}^n f_i(\mb{x}_i) - F^*$ (with $F^*$ obtained via centralized SVM) is also shown. } \label{fig_dynamics_remov}
\end{figure*}

Here, we recall an application in distributed support-vector-machines (D-SVM) from \cite{dsvm}. The cost function to be optimized locally is in logarithmic hinge loss form as
\begin{align}\label{eq_svm_smooth}
&f_i(\mb{x}_i)=\boldsymbol{\omega}_i^\top \boldsymbol{\omega}_i + C \sum_{j=1}^{N_i} \tfrac{1}{\mu}\log (1+\exp(\mu z_i)),\\
&{\mb{x}_i = [\boldsymbol{\omega}_i^\top;\nu_i]}\in\mathbb{R}^m,~~{z_i=1-l_j( \boldsymbol{\omega}_i^\top \phi(\boldsymbol{\chi}^i_j)-\nu_i)}
\end{align}
with $\phi(\cdot)$ as some proper nonlinear mapping. The simulation are done for the 3D binary classification example with $m=4$ and $50$ random data points. The loss function parameters are set as: $\mu=3$, $C=1.5$. The points are not linearly separable in 2D and transformed into 3D using proper nonlinear mapping in the form $\phi(\boldsymbol{\chi}^i_j)=[\boldsymbol{\chi}^i_j(1)^2;\boldsymbol{\chi}^i_j(2)^2;\sqrt{2}\boldsymbol{\chi}^i_j(1)\boldsymbol{\chi}^i_j(2)]$. The projected 3D points are linearly separable via SVM hyperplane and proper Kernel function. It should be clarified that every node has access to a different part of the SVM classification dataset. In fact, only $80\%$ of randomly chosen data points are assigned to each node.
The SC graph of $n=5$ nodes is shown in Fig.~\ref{fig_remov} (before and after link removal). For the sake of this simulation, the link weights are chosen randomly and \textit{symmetric}, and one of the links is removed to show the WB performance subject to link failure. In Fig.~\ref{fig_dynamics_remov}, the top row shows the algorithm evolution before link failure and the bottom row is for after link removal. We set $\alpha=1,\eta=0.001$ in Algorithm~\ref{alg_1} for this simulation and choose random initialization for D-SVM parameters. As it is clear, the proposed algorithm is still convergent under link failure. This shows how our work advances the state-of-the-art literature which 
assume stochastic link-weights. Note that stochastic weight design suffers from change in the network topology (including link failure), since redesign algorithms are needed to compensate for loss of stochasticity. Some of these compensation algorithms are proposed in \cite{6426252,cons_drop_siam,xu2018consensusability,li2019robust,gerencser2018push}. These algorithms add more computational complexity to the existing bi-stochastic weight design solutions.  Therefore, such algorithms cannot handle the link failure example in Fig.~\ref{fig_remov}, while our algorithm works with no additional computational complexity. Note that in Fig.~\ref{fig_dynamics_remov} the D-SVM parameters at all nodes reach consensus and converge to the centralized SVM parameters asymptotically, while each node has only access to a portion of the classification dataset. This is significant as it illustrates how our algorithm paves the way for \textit{distributed} and \textit{localized} machine learning solutions. 

\begin{figure} [bpt]
	\centering
		\includegraphics[width=3in]{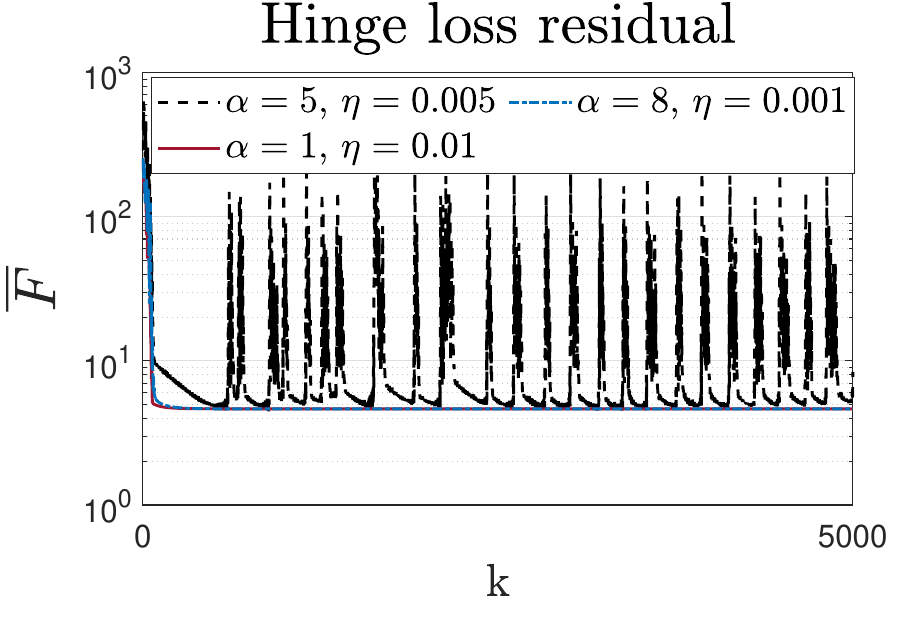}
	\caption{The D-SVM residual (at all $n=6$ nodes) under discrete dynamics~\eqref{eq_xydot}-\eqref{eq_M}, where the weight matrices are time-varying (random-weighting at every 100 iterations) over a 2-hop SC digraph. The hinge loss residual~$\overline{F}$ are compared for different $\alpha,\eta$ values. This example shows that for large values of both $\alpha=5,\eta=0.005$ solution is not stable.} \label{fig_dynamics}
\end{figure}

We repeat the simulation for a dynamic random network with time-varying link weights randomly changing every~$100$ iterations. This is to show that our setup can handle dynamic network topologies. The residuals $\overline{F}(\mb{x}) = \sum_{i=1}^n f_i(\mb{x}_i) - F^*$ (with $F^*$ obtained via centralized SVM)  are shown in Fig.~\ref{fig_dynamics} for different values of $\alpha,\eta$ in Algorithm~\ref{alg_1}: $\alpha=5,\eta=0.005$, $\alpha=8,\eta=0.001$, $\alpha=1,\eta=0.01$. As it can be seen there is a trade-off between these two variables, and by increasing one the other is decreased to satisfy the stability. For large values of both GT tracking parameter $\alpha$ and discrete step-size $\eta$ the stable convergence may not be achieved, e.g., see $\alpha=5$ and $\eta =0.005$. This is because (following the perturbation-based analysis in Section~\ref{sec_discrete}) for these values of $\alpha$ and $\eta$ one or more eigenvalues of the proposed dynamics~\eqref{eq_xydot_d} move to the RHP and causes instability. On the other hand, for other choices of $\alpha$ and $\eta$ satisfying Eq.~\eqref{eq_alphabar_eta} all eigenvalues remain in the LHP and stable convergence holds.

\section{Conclusions and Future Directions}
This work relaxes the weight constraint on the consensus optimization algorithms, e.g., to be applied over dynamic networks. One main application is in distributed mobile sensor networks and learning over coordinated (swarm) robotic systems, where the connectivity mainly follows the nearest neighbour rule \cite{jadbabaie2003coordination} or (limited) broadcasting range of the mobile entities (or agents) with the links coming and going as their formation evolves. As one direction of future research, one can further address possible communication time-delays or link failure  (similar to the equality-constraint optimization in \cite{scl,9663474}). Also, actuation and data-transmission nonlinearities on the node dynamics, e.g., in terms of quantization or saturation for equality-constraint optimization  \cite{doostmohammadian20211st,rikos2021optimal}, can be extended to this current work.

\section*{Acknowledgements}
The authors would like to thank Themistoklis Charalambous and Usman A. Khan for their help and comments.

\bibliographystyle{IEEEbib}
\bibliography{bibliography}

\begin{IEEEbiography}[{\includegraphics[width=1.1in,clip,keepaspectratio]{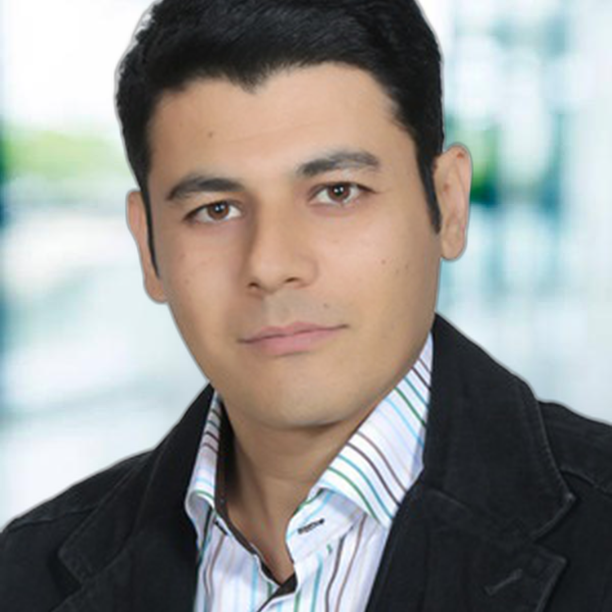}}]{Mohammadreza~Doostmohammadian}
	received his B.Sc. and M.Sc. in Mechanical Engineering from Sharif University of Technology, Iran, respectively in 2007 and 2010, where he worked on applications of control systems and robotics. He received his PhD in Electrical and Computer Engineering from Tufts University, MA, USA in 2015. During his PhD in Signal Processing and Robotic Networks (SPARTN) lab, he worked on control and signal processing over networks with applications in social networks. From 2015 to 2017 he was a postdoc researcher at ICT Innovation Center for Advanced Information and Communication Technology (AICT), School of Computer Engineering, Sharif University of Technology, with research on network epidemic, distributed algorithms, and complexity analysis of distributed estimation methods. He was a researcher at Iran Telecommunication Research Center (ITRC), Tehran, Iran in 2017 working on distributed control algorithms and estimation over IoT. Since 2017 he has been an Assistant Professor with the Mechatronics Department at Semnan University,  Iran, and he was a researcher at the School of Electrical Engineering and Automation, Aalto University, Finland. His general research interests include distributed estimation, control, and convex optimization over networks. He was the chair of the robotics and control session at the ISME-2018 conference and also the session chair at 1st Artificial Intelligent Systems Conference of Iran, 2022.
	%He is a reviewer for IFAC and IEEE journals and conferences.
\end{IEEEbiography}

\begin{IEEEbiography}[{\includegraphics[width=1.1in]{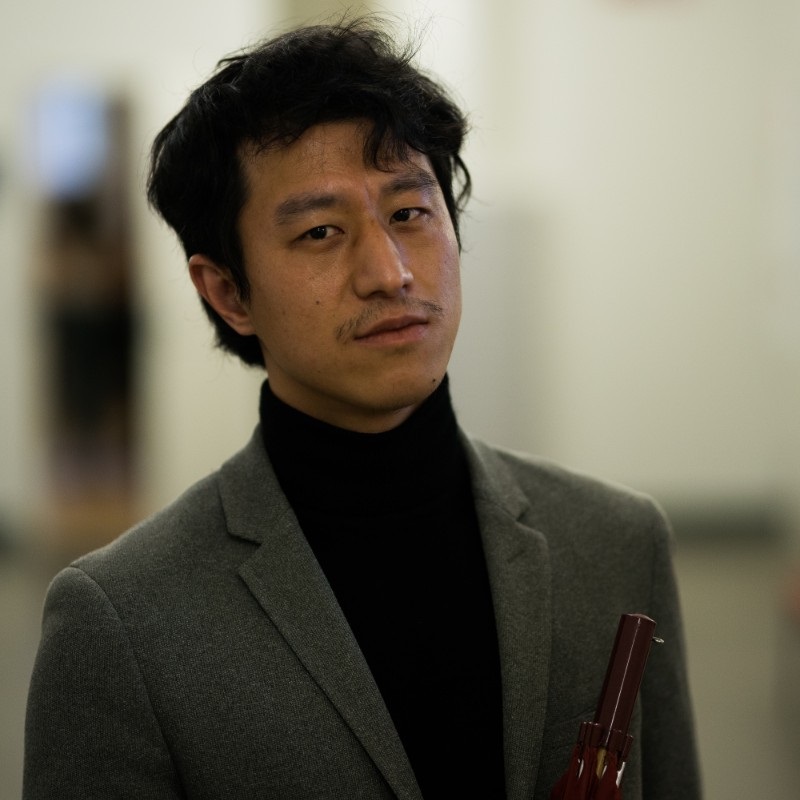}}]{Wei Jiang}
	got his Ph. D. degree from Ecole Centrale de Lille, France, working on cooperative control for multi-agent/robot systems and motion planning for mobile manipulator systems. He is a Postdoctoral researcher at the School of Electrical Engineering and Automation, Aalto University, Finland.
\end{IEEEbiography}

\begin{IEEEbiography}[{\includegraphics[width=1.1in]{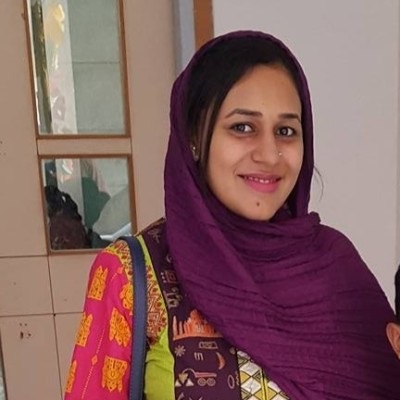}}]{Muwahida Liaquat}
	received her PhD in Electrical Engineering in 2013 from the National University of Sciences and Technology (NUST), Pakistan, specializing in Control Systems. Her work was focused on developing a reconstruction filter for the sampled data regulation of linear systems. She received her MS in Electrical Engineering in 2006, and her BE in Computer Engineering in 2004 from NUST. She is currently working as an Assistant Professor in the Department of Electrical Engineering at NUST College of Electrical and Mechanical Engineering, where she is part of the Control Systems research group along with teaching specialized courses in control engineering to both undergraduate and postgraduate levels. She has supervised a number of MS students in the area of control engineering. She was a Postdoctoral researcher at the School of Electrical Engineering and Automation, Aalto University, Finland. She is currently 
	a senior research scientist with Finnish Geospatial Research Institute (FGI).
\end{IEEEbiography}

\begin{IEEEbiography}[{\includegraphics[width=1.1in]{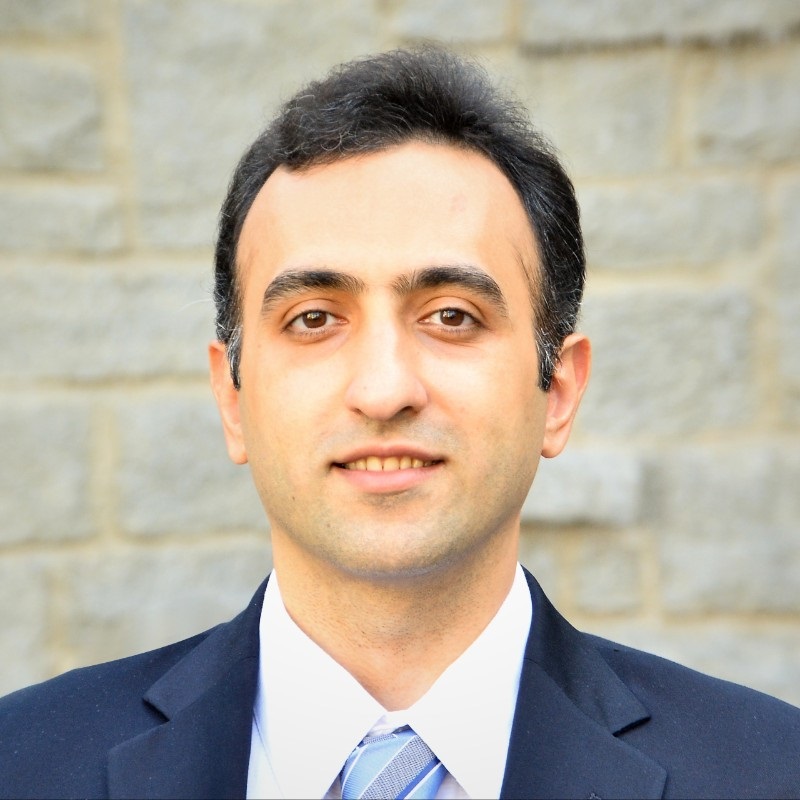}}]{ Alireza Aghasi}
	joined the School of Electrical Engineering and Computer Science at Oregon State University in Fall 2022. Between 2017 and 2022, he was an assistant professor in the Department of Data Science and Analytics at the Robinson College of Business, Georgia State University. Prior to this position, he was a research scientist with the Department of Mathematical Sciences, IBM T.J. Watson Research Centre, Yorktown Heights. From 2015 to 2016 he was a postdoctoral associate with the computational imaging group at the Massachusetts Institute of Technology, and between 2012 and 2015 he served as a postdoctoral research scientist with the compressed sensing group at Georgia Tech. He got his PhD from Tufts University in 2012. His research fundamentally focuses on optimization theory and statistics, with applications to various areas of data science, artificial intelligence, modern signal processing and physics-based inverse problems.
\end{IEEEbiography}

\begin{IEEEbiography}[{\includegraphics[width=1.1in]{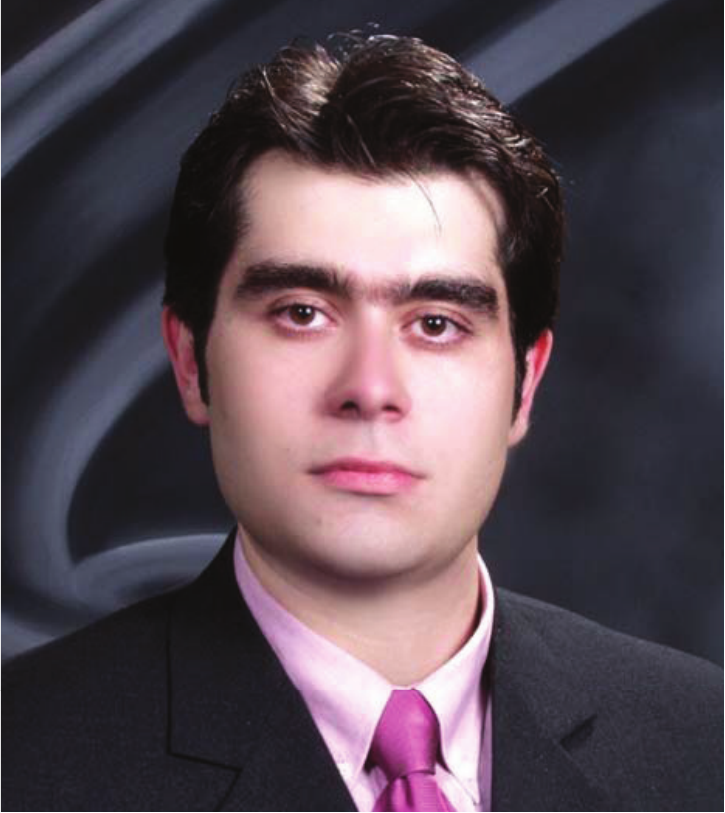}}]{Houman Zarrabi}
	received his PhD from Concordia University in Montreal, Canada in 2011. Since then he has been involved in various industrial and research projects. His main expertise includes IoT, M2M, CPS, big data, embedded systems, and VLSI. He is the national IoT program director and Assistant Professor at the Iran Telecommunication Research Center (ITRC).
\end{IEEEbiography}
\end{document}